\documentclass{article}
\usepackage{graphicx}
\usepackage{amsmath}
\usepackage{amsfonts}
\usepackage{amssymb}
\usepackage{cite}
\usepackage{setspace}
\usepackage{epsfig}
 \fontsize{11}{14}
\newtheorem{theorem}{Theorem}[section]

\newtheorem{conjecture}{Conjecture}[section]
\newtheorem{corollary}[theorem]{Corollary}

\newtheorem{lemma}[theorem]{Lemma}

\newtheorem{proposition}[theorem]{Proposition}
\newtheorem{remark}[theorem]{Remark}

\newenvironment{proof}[1][Proof]{\textbf{#1.} }{\ \rule{0.5em}{0.5em}}
\begin{document}
\title{Global Properties of Graphs with Local Degree Conditions}
\author{EWA KUBICKA, GRZEGORZ KUBICKI\\
\small{Department of Mathematics}\\
\small{University of Louisville, KY, USA}\\
\small{ewa@louisville.edu, grzegorz.kubicki@louisville.edu}\\
ORTRUD R. OELLERMANN\footnote{Supported by an NSERC grant CANADA,  Number 198281-2011}\\
 \small{Department of Mathematics and Statistics}\\
 \small{University of Winnipeg, Winnipeg MB, CANADA}\\
 \small{o.oellermann@uwinnipeg.ca}}
\date{}

\maketitle

\begin{abstract}
Let $\cal P$ be a graph property. A graph $G$ is said to be {\em locally} $\cal P$ (closed locally $\cal P$, respectively) if the subgraph induced by the open neighbourhood (closed neighbourhood, respectively) of every vertex in $G$ has property $\cal P$.  A graph $G$ of order $n$ is said to satisfy {\em Dirac's condition} if $\delta(G) \ge n/2$ and it satisfies {\em Ore's condition} if for every pair $u,v$ of non-adjacent vertices in $G$, $deg(u) + deg(v) \ge n$. A graph is {\em locally Dirac}  ({\em locally Ore}, respectively) if the subgraph induced by the open neighbourhood of every vertex satisfies Dirac's condition (Ore's condition, respectively). In this paper we establish global properties for graphs that are locally Dirac and locally Ore. In particular we show that these graphs, of sufficiently large order, are $3$-connected. For locally Dirac graphs it is shown that the  edge connectivity  equals the minimum degree and it is illustrated that this results does not extend to locally Ore graphs.  We show that $\lfloor n/3 \rfloor -1$ is a sharp upper bound on the diameter of every locally Dirac graph of order $n$.  We show that there exist infinite families of planar closed locally Dirac graphs. In contrast, locally Dirac graphs of sufficiently large order are shown to be non-planar.  It is known that every closed locally Ore graph is hamiltonian. We show that locally Dirac graphs have an even richer cycle structure by showing that all locally Dirac graphs with maximum degree 11 are in fact fully cycle extendable.   This result supports Ryj\'{a}\v{c}ek's well-known conjecture; which states that every connected, locally connected graph is weakly pancyclic.  \\

\medskip

\end{abstract}

\section{Introduction}
The development of graph theory has been profoundly influenced by the evolution of the internet and resulting large communication networks. Of particular interest are global properties of social networks, such as facebook, that can be deduced from their local properties. In this paper we investigate global properties in graphs that satisfy certain local degree conditions.

We begin by defining graph properties and invariants that we shall consider. Let $G$ be a graph. The order (number of vertices) of  $G$ is denoted by $n(G)$ or $n$ if $G$ is clear from context. The {\em diameter} of a connected graph $G$ is the maximum distance between all pairs of vertices of $G$. The {\em connectivity}, $\kappa(G)$ of $G$, is the minimum number of vertices of $G$ whose deletion from $G$ produces a disconnected graph or the trivial graph. The {\em edge-connectivity}, $\lambda(G)$ of $G$, is the minimum number of edges of $G$ whose deletion from $G$ produces a disconnected graph or the trivial graph. A graph $G$ is \emph{hamiltonian} if $G$ has a cycle of length $n(G)$. If, in addition, $G$ has a cycle of every length from 3 up to $n(G)$, then $G$ is \emph{pancyclic}. An even stronger notion than pancyclicity is that of full cycle extendability, introduced by Hendry \cite{H1}. A cycle $C$ in a graph $G$ is {\em extendable} if there exists a cycle $C'$ in $G$ that contains all the vertices of $C$ plus a single new vertex.  A graph $G$ is {\em cycle extendable} if every nonhamiltonian cycle of $G$ is extendable. If, in addition, every vertex of $G$ lies on a 3-cycle, then $G$ is \emph{fully cycle extendable.}

 Recall that the {\em girth}, denoted by $g(G)$, is defined as the length of a shortest cycle and the {\em circumference}, denoted by $c(G)$, is the length of a longest cycle in a graph $G$. A graph $G$ is called \emph{weakly pancyclic} if $G$ has a cycle of every length between $g(G)$ and $c(G)$.

By a local property of a graph we mean a property that is shared by the subgraphs induced by the open neighbourhoods of the vertices.
The {\em open neighbourhood} of a vertex $v\in V(G)$ is denoted by $N(v)$ and the {\em closed neighbourhood} of $v$, denoted by $N[v]$ is the set $N(v) \cup \{v\}$. If $X\subseteq V(G)$, the subgraph induced by $X$ is denoted by $\langle X\rangle$. For a given graph property $\cal P$, we call a graph $G$ \emph{locally $\cal P$} if $\langle N(v) \rangle$ has property $\cal P$ for every $v\in V(G)$. Skupie\'{n} \cite{S2} defined a graph $G$ to be \emph{locally hamiltonian} if $\langle N(v) \rangle$ is hamiltonian for every $v\in V(G)$.  Locally hamiltonian graphs were further studied in \cite{Pa,PS,S}. Pareek and Skupie\'{n} \cite{PS} considered locally traceable graphs and  Chartrand and Pippert \cite{CP} introduced locally connected graphs. The latter have since been studied extensively - see for example \cite{CGP,CP,C,GOPS,H1,H2,OS}. A graph is {\em closed locally} $\cal P$ if $\langle N[v] \rangle$ has property $\cal P$ for every $v\in V(G)$.

The minimum and maximum degree of a graph $G$ is denoted by $\delta(G)$ and $\Delta(G)$, respectively. If $G$ is clear from context we use $\delta$ and $\Delta$, instead. For notation and definitions not included here we refer the reader to \cite{bm}.

A classic example of a local property that guarantees hamiltonicity is Dirac's minimum degree condition (see \cite{D}).

\begin{theorem} \emph{\cite{D}} \label{Diracs_condition}
Let $G$ be a graph of order $n \ge 3$. If $\delta(G) \ge n/2$, then $G$ is hamiltonian.
\end{theorem}

Thus Dirac's condition may be written as `$|N(v)|\geq n(G)/2$ for every vertex $v$ in $G$'. Bondy \cite{B} showed that Dirac's minimum degree condition actually guarantees more than just the existence of a Hamilton cycle.

\begin{theorem} \emph{\cite{B}} If $G$ is a graph such that $\delta(G) \geq n(G)/2$, then $G$ is either pancyclic or isomorphic to the complete, balanced bipartite graph $K_{n/2,n/2}$.
\end{theorem}

A weaker degree condition that guarantees a graph to be hamiltonian is due to Ore \cite{Ore}.

\begin{theorem} \emph{\cite{Ore}} \label{Ores_condition}
Let $G$ be a graph of order $n$. If $deg_G(u) +deg_G(v) \ge n$ for every pair $u,v$ of non-adjacent vertices of $G$, then $G$ is hamiltonian.
\end{theorem}

Another local property that is often studied in connection with hamiltonicity is the property of being claw-free, i.e., not having the claw $K_{1,3}$ as induced subgraph. Note that a graph $G$ is claw-free if and only if $\alpha (\langle N(v)\rangle)\leq 2$ for every $v\in V(G)$ (where $\alpha$ denotes the vertex independence number).

 It is well known that the {\em Hamilton Cycle Problem} (the problem of deciding whether a graph has a Hamiltonian cycle) is NP-complete, even for claw-free graphs. The following well-known theorem of Oberly and Sumner \cite{OS}, demonstrates the strength of the local connectivity property.

\begin{theorem} \label{oberly and sumner} \emph{\cite{OS}}
If $G$ is a connected, locally connected, claw-free graph of order at least 3, then $G$ is hamiltonian.

\end{theorem}

Clark \cite{C} strengthened Theorem \ref{oberly and sumner} by showing that if $G$ is a connected, locally connected, claw-free graph, then $G$ is pancyclic. Subsequently Hendry \cite{H1} showed that under the same conditions the graph is in fact fully cycle extendable.
These results support Bondy's well-known `meta-conjecture' (see \cite{B2}) that almost any condition that guarantees that a graph has a Hamilton cycle actually guarantees much more about the cycle structure of the graph.

If, in Theorem \ref{oberly and sumner}, the claw-free condition is dropped, hamiltonicity is no longer guaranteed. In fact, Pareek and Skupie{\'n} \cite{PS} observed that there exist infinitely many connected, locally hamiltonian graphs that are nonhamiltonian. However, Clark's result led Ryj\'{a}\v{c}ek to suspect that every locally connected graph has a rich cycle structure, even if it is not hamiltonian. He proposed the following conjecture (see \cite{WR}.)

 \begin{conjecture} \label{ryjacek}
 (Ryj\'{a}\v{c}ek) Every locally connected graph is weakly pancyclic.
 \end{conjecture}

 Ryj\'{a}\v{c}ek's conjecture seems to be very difficult to settle. Several conditions stronger than local connectedness have been imposed on graphs to obtain results in support of Ryj\'{a}\v{c}ek's conjecture.  Nevertheless, it often remains a difficult problem to decide which of these graphs  are hamiltonian. For example, locally hamiltonian graphs introduced by Skupie\'{n} \cite{S2} need not be hamiltonian. It is shown, for example, in \cite{ad} that there exist infinitely many locally hamiltonian graphs that are not hamiltonian. Moreover, there does not appear to be an easy way of recognizing which locally hamiltonian graphs are in fact hamiltonian. The class of `locally isometric graphs' introduced in \cite{bno}, is a class of graphs satisfying another such local condition. A subgraph $H$ of a graph $G$ is {\em isometric} if $d_H(u,v)=d_G(u,v)$ for all $u,v \in V(H)$.  A graph $G$ is {\em locally isometric} if the subgraph induced by the open neighbourhood of every vertex in $G$ is an isometric subgraph of $G$. It was shown in \cite{bno} that the problem of deciding whether a locally isometric graph is hamiltonian is NP-complete for graphs with maximum degree at most 8.  Locally connected graphs that are sufficiently `locally dense' were introduced in \cite{bno2}. The {\em clustering coefficient} of a vertex in a graph is the proportion of pairs of neighbours of the vertex that are themselves neighbours (see \cite{ws}). The {\em minimum clustering coefficient} of a graph $G$ is the smallest clustering coefficient of its vertices, taken over all  vertices (see \cite{bno2}). It was shown in \cite{bno2}, that even for connected locally connected graphs with minimum clustering coefficient as large as 1/2, hamiltonicity of the graph is not guaranteed. Nevertheless, it was shown that many of these graphs have a rich cycle structure. At the intersection of the locally hamiltonian, locally isometric, and locally connected graphs with minimum clustering coefficient at least $1/2$, lie the `locally Dirac' and 'locally Ore' graphs. We say that a graph $G$ is {\em locally Dirac} if for every $v \in V(G)$, $deg_{\langle N(v) \rangle}(u) \ge deg_G(v)/2$ for all $u \in N(v)$, i.e., the subgraph $\langle N(v) \rangle$ satisfies Dirac's condition for all $v \in V(G)$. Similarly, a graph $G$ is {\em locally Ore} if for every $v \in V(G)$, $deg_{\langle N(v) \rangle}(u) + deg_{\langle N(v) \rangle}(w)\ge deg_G(v)$ for all pairs $u,w$ of non-adjacent vertices in $N(v)$. In contrast with graphs satisfying the Dirac or Ore conditions, we will show that the locally Dirac and Ore graphs may be sparse and yet possess many of the nice properties that graphs with the Dirac and Ore conditions possess.

 Hasratian and Khachatrian in \cite{HK}  showed that if $G$ is closed locally Ore, i.e., if the subgraph induced by the closed neighbourhood of every vertex of $G$ satisfies Ore's condition, then the graph is hamiltonian.

 \begin{theorem} \emph{\cite{HK}} \label{closed_locally_Ore}
 Let $G$ be a graph of order $n \ge 3$. If $\langle N[v] \rangle$ satisfies Ore's condition for all $v \in V(G)$, then $G$ is hamiltonian.
 \end{theorem}

 \begin{remark} \label{1_2_extendable}
 The proof of Theorem \ref{closed_locally_Ore} given in \cite{HK} in fact shows that if $G$ is closed locally Ore and $C$ is a non-hamiltonian cycle, then there exists a cycle $C'$ of length 1 or 2 greater than $C$ that contains the vertices of $C$. Graphs with this property are called $\{1,2\}$-{\em extendable}.
 \end{remark}

 As an immediate consequence we obtain the following.

 \begin{corollary} \label{locally_Ore}
 Let $G$ be a graph of order $n \ge 3$. If for every $v \in V(G)$ and for all $u,w \in N(v)$, $deg_{\langle N(v) \rangle}(u) + deg_{\langle N(v) \rangle}(w) \ge deg_G(v)$, then $G$ is hamiltonian and $\{1,2\}$-extendable.
 \end{corollary}
 \begin{proof} Let $v \in V(G)$ and $u,w \in N(v)$. Since $deg_{\langle N[v] \rangle}(u) =deg_{\langle N(v) \rangle}(u)+1$ and $deg_{\langle N[v] \rangle}(w) =deg_{\langle N(v) \rangle}(w)+1$, it follows that $deg_{\langle N[v] \rangle}(u)+deg_{\langle N[v] \rangle}(w)=deg_{\langle N(v) \rangle}(u) + deg_{\langle N(v) \rangle}(w)+2\ge deg_G(v)+2 = |N[v]|+1 > |N[v]|$. Hence $\langle N[v] \rangle$ satisfies Ore's condition for all $v \in V(G)$. By Theorem \ref{closed_locally_Ore} we see that $G$ is hamiltonian and, by Remark \ref{1_2_extendable}, $G$ is $\{1,2\}$-extendable.
 \end{proof}

 The following is another consequence of this result.

 \begin{corollary} \label{locally_Dirac}
 Let $G$ be a graph of order $n \ge 3$. If for every $v \in V(G)$ and for all $u \in N(v)$, $deg_{\langle N(v) \rangle}(u)  \ge deg_G(v)/2$, then $G$ is hamiltonian and $\{1,2\}$-extendable.
 \end{corollary}

 The {\em strong product} of two graphs $G$ and $H$, denoted by $G \boxtimes H$,  is the graph with vertex set $V(G \boxtimes H)=V(G) \times V(H)$ and edge set $E(G \boxtimes H) =\{(u,v)(x,y)|~u=x ~and ~vy \in E(H)\} \cup \{(u,v)(x,y)|~v=y~ and~ ux \in E(G)\} \cup \{(u,v)(x,y) |~ux \in E(G)~ and ~ vy \in E(H)\}$.

 The {\em join} of two graphs $G$ and $H$, denoted by $G+H$ is the graph with vertex set $V(G) \cup V(H)$ and edge set $E(G) \cup E(H) \cup \{uv| u \in V(G)~ and ~v \in V(H)\}$.

 Let $u$ and $v$ be vertices of a graph $G$. Then $u \sim v$ is used to indicate that $u$ is adjacent with $v$ and $u \nsim v$ is used to indicate that $u$ is not adjacent with $v$.

\section{Connectedness and Diameter in Locally Ore and Dirac Graphs}

It is easily seen that the diameter of graphs satisfying the Dirac or the Ore condition is at most 2. However, graphs that are locally Dirac can have arbitrarily large diameter.  To see this let $P_m$ be the path of order $m$, $C_m$ be the cycle of order $m$ and $K_3$ the complete graph of order $3$. Then $P_m \boxtimes K_3$ is a locally Dirac graph of order $3m$ and diameter $m-1$ and $C_m \boxtimes K_3$ is  a locally Dirac graph of order $3m$ and diameter $\lfloor m/2 \rfloor$.  Graphs that satisfy the Dirac (or Ore) condition may not be locally Dirac (locally Ore, respectively). For example, for even $n \ge 4$, the complete bipartite graph $K_{n/2,n/2}$ satisfies the Dirac condition (as well as the Ore condition) but it is not locally Dirac (nor locally Ore). However, there are graphs such as regular complete $k$-partite graphs for $k \ge 3$ or the $k^{th}$ power of the cycle $C_{3k}$ for some $k \ge 1$, that satisfy the Dirac condition and are locally Dirac.

One may well ask whether the locally Dirac graphs can be characterized in terms of forbidden (induced) subgraphs. The next results shows that this is not the case.

\begin{proposition} \label{noforbidden_induced_subgraphs}
Every connected graph $G$ of order $n \ge 3$ is an induced subgraph of a locally Dirac graph.
\end{proposition}
\begin{proof} Let $H=G + K_n$. Then $H$ is a locally Dirac graph that contains $G$ as an induced subgraph.
\end{proof}

\medskip

The next result gives a sharp lower bound on the connectivity of a connected locally Dirac graph.

\medskip

\begin{theorem} \label{connectivity_loc_dirac}
If $G$ is connected locally Ore graph of order $n \ge 4$, then $G$ is $3$-connected.
\end{theorem}
\begin{proof} It is readily seen that a connected locally Ore graph of order at least $4$ is $2$-connected.
Suppose, to the contrary, that $G$ has a vertex-cut $S$ of cardinality 2, where $S = \{ u, v \}$. Let $ C_1, C_2,  \ldots, C_k$, $k \ge 2$, be the components of $G-S$.  Consider the sets $N(v) \cap V(C_i)$ and let $d = deg_G(v)$. Observe that each of these sets is non-empty otherwise $u$ is a cut-vertex of $G$. Let $x \in N(v) \cap V(C_1)$ and $y \in N(v) \cap V(C_2)$. We consider two cases.\\
\textbf{Case 1.} If $uv \in E(G)$, then $deg_{\langle N(v) \rangle}(x) \le (|N(v) \cap V(C_1)|-1)+1=|N(v) \cap V(C_1)|$. Also $deg_{\langle N(v) \rangle}(y) \le |N(v) \cap V(C_2)| \le d-|N(v) \cap V(C_1)|-1$. So $deg_{\langle N(v) \rangle}(x)+deg_{\langle N(v) \rangle}(y) < d$, a contradiction.

\noindent\textbf{Case 2.}  If $uv \notin E(G)$, then $d \ge |N(v) \cap V(C_1)| + |N(v) \cap V(C_2)|$.
However, $deg_{\langle N(v) \rangle}(x) \le |N(v) \cap V(C_1)| - 1$ and $deg_{\langle N(v) \rangle}(y) \le |N(v) \cap V(C_2)| - 1$. So $deg_{\langle N(v) \rangle}(x)+deg_{\langle N(v) \rangle}(y) < d$, a contradiction.
\end{proof}

\medskip

An immediate consequence of the previous result now follows.

\medskip

\begin{corollary}
If $G$ is a connected locally Dirac graph of order at least $4$, then $G$ is $3$-connected.
\end{corollary}

\medskip

To see that the bound in the previous two results is sharp, observe that the graph $P_m \boxtimes K_3$, for $m \ge 3$, is a connected locally Ore/Dirac graph with connectivity $3$. If we add a new vertex to $P_m \boxtimes K_3$ and join it to three pairwise adjacent vertices of degree $5$ in $P_m \boxtimes K_3$, we obtain a locally Ore graph with minimum degree $3$.  In the next result shows that three cannot be the minimum degree of locally Dirac graphs of sufficiently large order.

\begin{theorem} \label{mindegree_loc_dirac}
If $G$ is a connected locally Dirac graph or order $n \ge 8$, then $\delta(G) \ge 5$.
\end{theorem}
\begin{proof}
Since $n \ge 8$, it follows from Theorem \ref{connectivity_loc_dirac} that $\delta(G) \ge 3$. Let $v$ be a vertex of degree $\delta(G)$ and let $N_2(v)$ consist of all vertices distance exactly $2$ from $v$. If $\delta(G) < 5$, then $\delta(G)=3$ or $4$.

Assume first that $\delta(G) =3$ and let $N(v) =\{x,y,z\}$. Since $G$ is locally Dirac, $N(v)$ induces a $K_3$. By Theorem \ref{connectivity_loc_dirac} every vertex of $N(v)$ is adjacent with at least one vertex of $N_2(v)$. If some vertex of $N(v)$, say $x$ is adjacent with at least two vertices of $N_2(v)$, then it follows, since $G$ is locally Dirac, that $deg_{\langle N(x) \rangle}(v) \ge \lceil 5/2 \rceil =3$. This is not possible since $v$ has at most two neighbours in $\langle N(x) \rangle$. So each vertex of $N(v)$ is adjacent with exactly one vertex in $N_2(v)$. Let $u$ be a neighbour of $x$ in $N_2(v)$. Since $G$ is locally Dirac, $u$ must be adjacent with both $y$ and $z$. But then $G$ has order $5$, a contradiction. So $\delta(G) \ne 3$.

Assume next that $\delta(G) =4$. Let  $N(v) = \{v_1, v_2, v_3, v_4\}$. Since $n \ge 8$ and by Theorem \ref{connectivity_loc_dirac} we must have $|N_2(v)| \ge 3$. Suppose first that each vertex from $N(v)$ is adjacent to at most one vertex from $N_2(v)$. Then there is a vertex $a \in N_2(v)$ such that $a$ is adjacent to exactly one vertex of $N(v)$; otherwise each vertex from $N_2(v)$  has at least two neighbours in $N(v)$, which contradicts our assumption that  each vertex from $N(v)$ is adjacent to at most one vertex from $N_2(v)$. We may assume that $v_1a \in E(G)$ and that $a$ is not adjacent to any of $v_2, v_3, v_4$. Then $deg_{\langle N(v_1) \rangle}(a) = 0$. Since $G$ is locally Dirac and $|N(v_1)| \ge 4$, this is not possible. Therefore, there is a vertex in $N(v)$, say $v_1$, that is adjacent to at least two vertices in $N_2(v)$, say $a$ and $b$. Then $|N(v_1)| \ge 5$, and thus $deg_{\langle N(v_1) \rangle} (v) \ge 3$, which implies that $v_1$ is adjacent with every vertex of $\{v_2,v_3, v_4\}$. So $|N(v_1)| \ge 6$.  The vertex $v_1$ cannot be adjacent to any other vertices, because $|N(v_1)| \ge 7$ would imply  $deg_{\langle N(v_1) \rangle} (v) \ge 4$, which is impossible. Similarly each of $v_2, v_3$, and $v_4$ is adjacent to at most two vertices in $N_2(v)$. This implies, since $|N_2(v)| \ge 3$, that there is a vertex in $N_2(v)$ adjacent to at most two vertices from $N(v)$. Suppose that $z$ is such a vertex and that $zv_i \in E(G)$, for some $i, 1 \le i \le 4$. If $v_i$ has no other neighbours in $N_2(v)$ except $z$, then $|N(v_i)| \ge 4$ but  $deg_{\langle N(v_i) \rangle} (z) \le 1$, so $G$ is not locally Dirac.  If $v_i$ has another neighbour in $N_2(v)$, then $|N(v_i)| \ge 5$, but   $deg_{\langle N(v_i) \rangle} (z) \le 2$, so $G$ is not locally Dirac.
 \end{proof}

\begin{remark}
There are infinitely many planar closed locally Dirac graphs. For example, the graphs $P_m \boxtimes K_2$, for $m \ge 3$, forms such a class of graphs.
\end{remark}

\medskip

For Locally Dirac graphs the situation is different as our next result shows. We will use the result established in \cite{CP} which states that every locally $3$-connected graph is non-planar.

\begin{theorem} \label{loc_Dirac_non_planar}
Every locally Dirac graph of order $n \ge 8$ is non-planar.
\end{theorem}
\begin{proof} If $\langle N(v) \rangle$ is $3$-connected for all $v \in V(G)$, then the results follows from the above. Suppose now that $G$ contains a vertex $u$ such that $H=\langle N(u) \rangle$ is not $3$-connected. Since $\delta(G) \ge 5$ and as $H=\langle N(u) \rangle$  satisfies the Dirac condition, $H$ has a hamilton cycle and is thus $2$-connected. Let $S=\{x,y\}$ be a $2$-vertex cut of $H$. Let $H_1$ be a component of $H-S$ of smallest order. Then $H_1$ has at most $\frac{d-2}{2}$ vertices. Since $G$ is locally Dirac, the vertices of $H_1$ necessarily induce a complete graph and are all adjacent (in $H$ and hence in $G$) with every vertex of $S$ and have degree exactly $\frac{d}{2}$ in $H$. Hence $d$ is even. If $d \ge 8$, then the subgraph induced by any three vertices of $H_1$ and $S \cup \{u\}$ contains a $K_{3,3}$ as subgraph. So $G$ is non-planar. If $d=6$, then $H-S$ has two components both with two (adjacent) vertices. So $H$ contains a subdivision of $K_4$ which together with $u$ yields a subdivision of $K_5$. So $G$ is non-planar.
\end{proof}

\medskip

Recall that the {\em eccentricity} of a vertex $v$ in a connected graph $G$ is $e(v)= max\{d(v,u)| u \in V(G)\}$ and the {\em diameter} is the maximum eccentricity among all pairs of vertices. Our next result provides a sharp upper bound on the diameter of a locally Dirac graph.

\begin{theorem} \label{diameter}
If $G$ is a connected locally Dirac graph of order $n \ge 9$, then $diam(G) \le \lfloor\frac{n}{3}\rfloor -1$. Moreover this bound is sharp.
\end{theorem}
\begin{proof}
If $G$ has diameter at most $2$, the result follows. Suppose $G$ has diameter at least $3$. Let $v$ be a vertex of $G$ such that $e(v) = diam(G)=d$. For each $i$, $0 \le i \le d$, let $V_i$ be the set of all vertices distance $i$ from $v$. By Theorem \ref{mindegree_loc_dirac}, $|V_0 \cup V_1| \ge 6$ and $|V_{d-1} \cup V_d| \ge 6$. By Theorem \ref{connectivity_loc_dirac}, $|V_i| \ge 3$ for $1 \le i < d$. So $n-12 \ge 3(d-3)$, i.e. $d \le \frac{n}{3} -1$.

This bound  is sharp since the graph $G = P_m \boxtimes K_3$ of order $n = 3m$ satisfies the condition $diam(G) = \frac{n}{3}-1$.
\end{proof}

\begin{remark}
If $G$ is locally Ore, then $diam(G) \le \lfloor \frac{n+1}{3} \rfloor$. Moreover, this bound is attained for every integer $n \ge 9$. Observe that $n$ is of the form $3k$ or $3k+1$ or $3k+2$ for some integer $n \ge 3$. If $n=3k$ or $3k+1$, start by taking a copy of $P_{k-1} \boxtimes K_3$. This graph contains two sets $S_1$ and $S_2$ of disjoint $K_3$'s whose vertices all have degree 5 in $G$. If $n=3k$, join one new vertex to one of these two sets of vertices and a $K_2$ to the other set to produce a locally Ore graph with the desired diameter. If $n=3k+1$, join a $K_2$ to the vertices of $S_1$ and join another $K_2$ to the vertices in $S_2$. If $n=3k+2$, start by constructing a $P_k \boxtimes K_3$. Again let $S_1$ and $S_2$ denote two disjoint sets of vertices that induce a $K_3$ and have degree $5$ in $P_k \boxtimes K_3$. Now add two new vertices and join one of them to the vertices of $S_1$ and the other to the vertices of $S_2$. In each case the resulting graph is locally Ore with diameter $\lfloor \frac{n+1}{3} \rfloor$.
\end{remark}

\medskip

It is well-known that $\lambda(G) \le \delta(G)$ and Plesn\'{i}k \cite{Pl} showed that equality holds for graphs with diameter at most $2$.  We show that this is also the case for locally Dirac graphs but that this result does not extend to graphs that are locally Ore and hence not to graphs that are closed locally Ore.

\begin{theorem}\label{edge_connectivity_locally_Dirac}
If $G$ is a connected locally Dirac graph of order $n \ge 3$, then $\lambda(G) = \delta(G)$.
\end{theorem}
\begin{proof} It is readily seen that the only locally Dirac graphs of orders $3$ or $4$ are complete.  Moreover the only locally Dirac graphs of order $5$ are $K_5$ and $K_5-e$ where $e$ is any edge of the $K_5$. Thus $\lambda(G) = \delta(G)$ for $3 \le n \le 5$.

Let $G$ be a locally Dirac graph of order $n \ge 6$ and let $S$ be a minimum edge-cut of $G$. Let $G_1$ and $G_2$ be the two components of $G-S$. Among all vertices of $G-S$ incident with edges of $S$, let $v$ be one incident with a maximum number of edges of $S$. We may assume that $v$ belongs to $G_1$. Suppose $v$ is incident with $k$ edges of $S$. Thus each of these $k$ edges joins $v$ with a vertex of $G_2$.

Assume first that $k \ge deg(v)/2$. If $k = deg(v)$ the results follows from the above remark. Suppose now that $v$ is adjacent with vertices of $G_1$. Let $u$ be a neighbour of $v$ in $G_1$. Since there are $deg(v)-k$ neighbours of $v$ in $G_1$, the vertex $u$ is adjacent with at most $deg(v)-k-1 < deg(v)/2$ neighbours of $v$ in $G_1$. Hence $u$ must be adjacent with a neighbour $u'$ of $v$ in $G_2$. So $uu' \in S$. Thus $|S| \ge deg(v)$. Since $|S| \le \delta(G) \le deg(v)$ we see that $\lambda(G) = \delta(G)$.

Assume next that $k < deg(v)/2$. Let $u$ be a neighbour of $v$ in $G_2$. Since $G$ is locally Dirac and since $u$ is adjacent with at most $k-1$ neighbours of $v$ in $G_2$, it follows that $u$ is adjacent with at least $\frac{deg(v)}{2} -k +1$ neighbours of $v$ in $G_1$. So $S$ contains at least $k(\frac{deg(v)}{2}-k+1) +k$ edges. Hence  $k(\frac{deg(v)}{2}-k+1) +k \le deg(v)$. So $(k-2)\frac{deg(v)}{2} \le k(k-2)$. If $k \ge 3$, we get $\frac{deg(v)}{2} \le k$, contrary to our assumption. So $k=1$ or $k=2$. Suppose $k=1$. Let $u$ be the neighbour of $v$ in $G_2$. Since $k < \frac{deg(v)}{2}$, $v$ must have at least two neighbours in $G_1$, i.e., $deg(v) \ge 3$. Since $G$ is locally Dirac it follows that $u$ must have at least two neighbours in $V(G_1) \cap N(v)$. So $u$ is incident with at least three edges of $S$, contrary to our choice of $v$. So $k \ne 1$. Suppose $k =2$. Then $v$ has at least three neighbours in $G_1$. So $deg(v) \ge 5$. So $u$, a neighbour of $v$ in $G_2$, is adjacent with at least three neighbours of $v$ of which at least two are in $G_1$. So $u$ is incident with at least three edges of $S$, contrary to our choice of $v$.
\end{proof}

\medskip

We now show that this result does not extend to graphs that are locally Ore.

\begin{proposition}
There exist infinitely many graphs $G$ that are locally Ore and  such that $\lambda(G) \ne \delta(G)$.
\end{proposition}
\begin{proof} Let $k \ge 3$ be an integer. Let $G_{k,1}$ and $G_{k,2}$ be two copies of $K_{k^2+2}$ with vertex sets $\{v_1,v_2, \ldots, v_{k^2+2}\}$ and $\{u_1,u_2, \ldots, u_{k^2+2}\}$, respectively. Let $G_k$ be the graph obtained from $G_{k,1} \cup G_{k,2}$ by adding all edges between the set $\{v_1,v_2, \ldots, v_k\}$ and the set $\{u_1,u_2, \ldots, u_k\}$. Then $G_k$ is locally Ore and $\delta(G_k)=k^2+1$ but $\lambda(G_k) =k^2$.
\end{proof}

\section{Cycle Structure of Locally Dirac Graphs}
In this section we show that locally Dirac graphs with maximum degree at most 11 are fully cycle extendable. We begin with a few definitions, some notation and useful results.
Let $C=v_0v_1v_2\ldots v_{t-1} v_0$ be a $t$-cycle in a graph $G$. If $i \ne j$ and $\{i,j\}\subseteq \{0,1,\ldots, t-1\}$, then $v_i\overrightarrow{C}v_j$ and $v_i\overleftarrow{C}v_j$ denote, respectively, the paths $v_iv_{i+1}\ldots v_j$ and $v_iv_{i-1}\ldots v_j$ (subscripts expressed modulo $t$).  Let $C=v_0 v_1, \ldots v_{t-1} v_0$ be a non-extendable cycle in a graph $G$. With reference to a given non-extendable cycle $C$, a vertex of $G$ will be called a \emph{cycle vertex} if it is on $C$, and an \emph{off-cycle }vertex if it is in $V(G)-V(C)$. A cycle vertex that is adjacent to an off-cycle vertex will be called an \emph{attachment vertex}.
The following basic results on non-extendable cycles will be used frequently and were established in \cite{afow}. Since the proofs are short we include them here for completeness.

\begin{lemma} \cite{afow} \label{nonextendable1} Let $C=v_0v_1\ldots v_{t-1} v_0$ be a non-extendable cycle of length $t$ in a graph $G$. Suppose $v_i$ and $v_j$ are two distinct attachment vertices of $C$ that have a common off-cycle neighbour $x$. Then the following hold. (All subscripts are expressed modulo $t$.)

\begin{itemize}
\item[\emph{1}.] $j\neq i+1$.
\item[\emph{2}.] Neither $v_{i+1}v_{j+1}$ nor $v_{i-1}v_{j-1}$ is in $E(G)$.

\item[\emph{3}.] If $v_{i-1}v_{i+1}\in E(G)$, then neither $v_{j-1}v_i$ nor $v_{j+1}v_i$ is in $E(G)$.
\item[\emph{4}.] If $j=i+2$ then $v_{i+1}$ does not have two neighbours $v_k,v_{k+1}$ on the path  $v_{i+2}\ldots v_i$.
\end {itemize}
\end{lemma}

\begin{proof} We prove each item by presenting an extension of $C$ that would result if the given statement is assumed to be false. For (2) and (3) we only need to consider the first mentioned forbidden edge, due to symmetry.
\begin{enumerate}
\item[1.] $v_ixv_{i+1}\overrightarrow{C}v_i$.
\item[2.] $v_{i+1}v_{j+1}\overrightarrow{C}v_ixv_j\overleftarrow{C}v_{i+1}$.
\item[3.] $v_{j-1}v_ix v_j\overrightarrow{C}v_{i-1}v_{i+1}\overrightarrow{C}v_{j-1}$.
\item[4.] $v_kv_{i+1}v_{k+1}\overrightarrow{C}v_ixv_{i+2}\overrightarrow{C}v_k$.
\end{enumerate}
\hfill \end{proof}

\medskip

Before establishing the next main result we prove another useful lemma.

\medskip

\begin{lemma} \label{nonextendable2}
Let $C=v_0v_1\ldots v_{t-1} v_0$ be a non-extendable cycle of length $t$ in a connected locally Dirac graph $G$. Among all attachment vertices, select one of maximum degree. Assume that $v_0$ is such an attachment vertex with degree $d=deg(v_0)$ and suppose $v_0$ has $s \ge 1$ off-cycle neighbours. Let $x$ be an off-cycle neighbour of $v_0$.
\begin{enumerate}
\item[\emph{1}.] Then $d \ge 6$ and $s \le \frac{d}{2}-2$ if $v_1 \nsim v_{t-1}$ and $s \le \frac{d}{2}-1$ if $v_1 \sim v_{t-1}$.
\item[\emph{2}.] At least $\lceil \frac{s(\lceil d/2 \rceil-s+1)}{(d-s-2)} \rceil$ off-cycle neighbours of $v_0$ share a common cycle neighbour of $v_0$.
\item[\emph{3}.] If $v$ is a vertex of $G$, then every neighbour of $v$ has at most $\lfloor \frac{deg(v)}{2} \rfloor - 1$  non-neighbours in $\langle N(v) \rangle$ and if $v$ is an attachment vertex $v$ has at most $\lfloor \frac{d}{2} \rfloor -1$ non-neighbours in $\langle N(v) \rangle$.
\item[\emph{4}.] If an off-cycle neighbour $x$ is adjacent with $v_i$ and $v_{i+2}$ and some vertex $v_j$ on $v_{i+3} \overrightarrow{C}v_{i-2}$ is such that $v_j \sim \{v_{i+1},v_{i-1}\}$, then $v_{j-1} \nsim v_{j+1}$. Also if there is a $v_j$ on $v_{i+4} \overrightarrow{C}v_{i-1}$ such that $v_j \sim \{v_{i+1}, v_{i+3}\}$, then $v_{j-1} \nsim v_{j+1}$.
\item[\emph{5}.] If some off-cycle vertex $y$ is such that $y \sim\{v_i,v_j\}$ where $i < j$, then (i) there are no consecutive vertices on $v_j \overrightarrow{C}v_i$ such that one of these is adjacent with $v_{i+1}$ and the other with $v_{j-1}$, and (ii) there are no consecutive vertices on $v_i \overrightarrow{C} v_j$ such that one of them is adjacent with $v_{j+1}$ and the other with $v_{i-1}$.
\item[\emph{6}.] Suppose there exist vertices $v_i,v_j$ and $v_k$ on $C$ where $0 \le i < j-1$ and $j < k-1 < t-2$ and such that either (i) $x \sim\{v_i,v_j\}$, $v_{k-1}\sim v_{i+1}$, $v_{j+1} \sim v_{i-1}$, and $v_i \sim v_k$ or (ii) $x \sim\{v_i,v_k\}$, $v_{k-1} \sim v_{i+1}$, $v_{j+1} \sim v_{i-1}$ and $v_i \sim v_j$, or (iii) $x \sim \{v_i, v_j\}$, $v_{k+1} \sim v_{j-1}$, $v_{j+1} \sim v_{i-1}$ and $v_j \sim v_k$ or (iv) $x \sim \{v_i,v_k\}$, $v_{i+1} \sim v_{k-1}$, $v_{j-1} \sim v_{k+1}$ and $v_j \sim v_k$, then $C$ is extendable.
\item[\emph{7}.] If there is a vertex $v_j$ such that $2< j < t-2$ and $v_j \sim\{x,v_0,v_1\}$ or $v_j \sim \{x, v_0, v_{t-1}\}$, then $deg(v_0) \ge 8$.
\end{enumerate}
\end{lemma}
\begin{proof}
\begin{enumerate}
\item[1.]   Since $x$ is adjacent with at most $s-1$ off-cyle neighbours of $v_0$ it follows that $x$ is adjacent with at least $\frac{d}{2}-s+1$ cycle neighbours of $v_0$. By Lemma \ref{nonextendable1}(1), $x \nsim \{v_1,v_{t-1}\}$. So $\frac{d}{2}-s+1 \le d-s-2$. Hence $d \ge 6$.

    Suppose $v_1 \nsim v_{t-1}$. Since $v_1$ is not adjacent with any off-cycle neighbours of $v_0$, and since $v_1 \nsim v_{t-1}$, $d-s-2 \ge \frac{d}{2}$. Hence $s \le \frac{d}{2}-2$. If $v_1 \sim v_{t-1}$, then $v_1$ has at least $\frac{d}{2}-1$ neighbours that are cycle neighbours of $v_0$. So $s \le \frac{d}{2}-1$.
\item[2.] There are at least $s(\lceil d/2 \rceil-s+1)$ edges that join off-cycle neighbours of $v_0$ with the $d-s-2$ cycle neighbours of $v_0$ other than $v_1$ and $v_{t-1}$. So at least $s(\lceil d/2 \rceil-s+1)/(d-s-2)$ edges are incident with some cycle neighbour of $v_0$. Since $G$ has no multiple edges these edges are incident with distinct off-cycle neighbours of $v_0$.
\item[3.] This follows from the definition of a locally Dirac graph and our choice of $v_0$.
\item[4.] In the first case $v_{j-1}v_{j+1} \overrightarrow{C} v_{i-1}v_jv_{i+1}v_ixv_{i+2}\overrightarrow{C}v_{j-1}$ is an extension of $C$. The second case can be argued similarly.
\item[5.] (i) Suppose $v_{i+1} \sim v_{l}$ and $v_{j-1} \sim v_{l-1}$ for some $v_l$ and $v_{l-1}$ on $v_j \overrightarrow{C}v_i$. Then $v_iyv_j\overrightarrow{C}v_{l-1}v_{j-1} \overleftarrow{C}v_{i+1}v_l\overrightarrow{C}v_i$ is an extension of $C$.  Similarly if $v_{i+1} \sim v_{l-1}$ and $v_{j-1} \sim v_{l}$ for some $v_l$ and $v_{l-1}$ on $v_j\overrightarrow{C}v_i$, then $v_iyv_j\overrightarrow{C}v_{l-1}v_{i+1} \overrightarrow{C}$ $v_{j-1}v_l \overrightarrow{C}v_i$ is an extension of $C$. Case (ii) can be argued similarly.
\item[6.] In the case of $(i)$ $v_ixv_j \overleftarrow{C}v_{i+1}v_{k-1} \overleftarrow{C}v_{j+1}v_{i-1} \overleftarrow{C} v_kv_i$ is an extension of $C$ and in case (ii), $v_ixv_k \overrightarrow{C}$ $v_{i-1}v_{j+1} \overrightarrow{C} v_{k-1}$ $v_{i+1} \overrightarrow{C}v_jv_i$ is an extension of $C$. Cases (iii) and (iv) can be argued similarly.
\item[7.] Suppose $v_j \sim\{x,v_0,v_1\}$. By Lemmas \ref{nonextendable1} (1) - (3), $v_{j+1} \nsim\{x,v_1,v_{j-1}\}$. By part (3) above, $deg(v_j) \ge 8$. Hence $d=deg(v_0) \ge 8$. The case where $v_j \sim\{x,v_0,v_{t-1}\}$ can be argued similarly.
\end{enumerate}
\hfill \end{proof}

\medskip

 The next result shows that every locally Dirac graph with maximum degree at most 11 is not only Hamiltonian but in fact fully cycle extendable.

\begin{theorem} \label{cycle_extendability_in_loc_Dirac}
If $G$ is a connected locally Dirac graph with $\Delta(G)=\Delta \le 11$, then $G$ is fully cycle extendable.
\end{theorem}
\begin{proof}
Let $C=v_0v_1\ldots v_{t-1} v_0$ be a non-extendable cycle of length $t$ in a connected locally Dirac graph $G$. Among all attachment vertices, select one of maximum degree. Assume that $v_0$ is such an attachment vertex with degree $d=deg(v_0)$ and suppose $v_0$ has $s \ge 1$ off-cycle neighbours. Let $S$ be the collection of cycle neighbours of $v_0$ distinct from $v_1$ and $v_{t-1}$ and let $x$ be an off-cycle neighbour of $v_0$. By Lemma \ref{nonextendable1}(1), $x \nsim\{v_1,v_{t-1}\}$. So it follows from Lemma \ref{nonextendable2} (3) that $\frac{d}{2} -1 \ge 2$. So $\Delta \ge d \ge 6$.

\noindent{\bf Case 1} Suppose $d=6$.  Then, by Lemma \ref{nonextendable2} (1) every vertex in $N(v_0)$  is non-adjacent with at most two vertices in $\langle N(v_0) \rangle$, or equivalently, is adjacent with at least three vertices of $\langle N(v_0) \rangle$. By Lemma \ref{nonextendable1} (1), $x \nsim  \{v_1, v_{t-1}\}$. Let $S=N(v_0) - \{x, v_1, v_{t-1}\}$. If $v_1 \nsim v_{t-1}$, then it follows from the above that $\{x,v_1, v_{t-1}\} \sim S$. Since $|S|=3$, there is a $v_j \in N(v_0)$ such that $j \ne 2$ or $t-2$. By Lemmas \ref{nonextendable1} (1), (2) and (3), $v_{j+1} \nsim \{x, v_1, v_{j-1}\}$, contrary to Lemma \ref{nonextendable2} (1). If $v_1 \sim v_{t-1}$, then there is a $v_j \in S$ such that, $v_j \sim \{x, v_1\}$. By Lemma \ref{nonextendable1} (2), $j \not\in \{2,t-2\}$. As in the previous case we see that $v_{j+1}$ has at least three non-adjacencies in $\langle N(v_j) \rangle$, namely $v_{j+1} \nsim \{x, v_1, v_{j-1}\}$, contrary to Lemma \ref{nonextendable2} (1).

\noindent{\bf Case 2} Suppose $d=7$. Then every vertex in $N(v_0)$  is non-adjacent with at most two vertices in $\langle N(v_0) \rangle$, or equivalently, is adjacent with at least four vertices of $\langle N(v_0) \rangle$.  If $v_1 \nsim v_{t-1}$, $v_1$ is adjacent with at least four cycle neighbours of $v_0$ (different from $v_{t-1}$) and if $v_1 \sim v_{t-1}$, then both $v_1$ and $v_{t-1}$ are adjacent with at least three cycle neighbours of $v_0$. In either case there is a vertex $v_j$, where $ j \not\in \{2, t-2\}$, such that $v_j \sim \{x,v_0,v_1\}$. So, by Lemma \ref{nonextendable2} (7), $d \ge 8$.

\noindent{\bf Case 3} Suppose $d=8$. By Lemma \ref{nonextendable2} (3) each vertex of $N(v_0)$ is non-adjacent with at most three vertices of $N(v_0)$; so $v_0$ has at most three off-cycle neighbours.  Suppose $v_0$ has three off-cycle neighbours. Then $|S|=3$.  Since $v_1$ and $v_{t-1}$ are non-adjacent with every off-cycle neighbour of $v_0$ and since $G$ is locally Dirac, $\{v_1,v_{t-1}\} \sim S$ and $v_1 \sim v_{t-1}$.  Moreover, each off-cycle neighbour of $v_0$ is adjacent with at least two vertices of $S$. Hence $S$ contains a vertex $v_j$ that is adjacent with at least two off-cycle neighbours of $v_0$. By Lemma \ref{nonextendable1} (2), $j \ne 2$ and $j \ne t-2$. So $deg(v_j) \ge 7$. Since, by Lemmas \ref{nonextendable1} (1) and (2),  $v_{j+1}$ is not adjacent with the off-cycle neighbours of $v_j$ and $v_{j+1} \nsim v_1$ it follows, since $G$ is locally Dirac, and by our choice of $v_0$, that $v_{j+1}$ is adjacent with all other neighbours of $v_j$. So $v_{j+1} \sim v_{j-1}$. This contradicts Lemma \ref{nonextendable2} (3).

Suppose $v_0$ has exactly two off-cycle neighbours. Since each off-cycle neighbour of $v_0$ is adjacent with at least three cycle neighbours of $v_0$, there exist at least two vertices of $S$ that are adjacent with both off-cycle neighbours of $v_0$. Since $G$ is locally Dirac $v_1$ is adjacent with at least one of these vertices of $S$ that has two off-cycle neighbours in $N(v_0)$. Let $v_j$ be such a vertex. By Lemmas \ref{nonextendable1} (1), (2) and (3), $v_{j+1} \nsim \{v_1,v_{j-1}\}$ and $v_j$ is not adjacent with two off-cycle neighbours of $v_j$.  This is not possible unless $v_{j-1} = v_1$, i.e., $j=2$. By Lemma \ref{nonextendable1}(2) this implies that $v_1 \nsim v_{t-1}$. But now $v_1 \nsim \{v_{j+1}, v_{t-1}\}$ and $v_1$ is non-adjacent with the two off-cycle neighbours of $v_0$. This is not possible by Lemma \ref{nonextendable2} (3).

Suppose $v_0$ has exactly one off-cycle neighbour $x$. Since $G$ is locally Dirac,  $x$ has at least four neighbours in $S$ of which at least two are also neighbours of $v_1$. Let $v_j$ be such a common neighbour of $x, v_0$ and $v_1$ that is not $v_2$. By Lemmas \ref{nonextendable1} (1), (2) and (3), $v_{j+1} \nsim \{x,v_1,v_{j-1}\}$. So $v_{j+1} \sim v_0$, since $G$ is locally Dirac. Hence $x \sim (S-\{v_{j+1}\})$ and by Lemma \ref{nonextendable1} (3), $v_1 \nsim v_{t-1}$. But now there are at least three vertices of $S$ adjacent with both $x$ and $v_1$ of which at least two, say $v_j$ and $v_k$, are not $v_2$. By Lemma \ref{nonextendable1} (1), $\{v_{j+1},v_{k+1}\} \nsim x$ and since at least four vertices of $S$ are adjacent with $x$ either $v_{j+1}$ or $v_{k+1}$ is not adjacent with $v_0$, say the former. But now $v_{j+1}$ has at least four non-adjacencies in $\langle N(v_j) \rangle$, which is not possible.

\noindent{\bf Case 4} Suppose $d=9$. By Lemma \ref{nonextendable2} (3), each neighbour of an attachment vertex  has at most three non-neighbours. So $v_0$ has at most three off-cycle neighbours. Suppose $v_0$ has three off-cycle neighbours. Then $\{v_1, v_{t-1}\} \sim S$ and since each off-cycle neighbour has at least three neighbours in $S$, there is a vertex $v_j \in S$ such that $v_j$ is adjacent with all three off-cycle neighbours of $v_0$. By Lemmas \ref{nonextendable1} (1) and (3), $v_{j+1}$ is non adjacent with these three off-cycle neighbours of $v_j$ and $v_{j+1} \nsim v_{j-1}$, contrary to Lemma \ref{nonextendable2} (3). Suppose $v_0$ has two off-cycle neighbours. Since $G$ is locally Dirac, there are at least three vertices in $S$ that are adjacent with both off-cycle neighbours of $v_0$. Of these at least two are adjacent with $v_1$ and among these at least one, call it $v_j$,  is not $v_2$. So, by Lemmas \ref{nonextendable1} (1), (2) and (3), $v_{j+1} \nsim \{v_1,v_{j-1}\}$ and $v_{j+1}$ is not adjacent with both off-cycle neighbours of $v_0$, contrary to Lemma \ref{nonextendable2} (3).

\noindent{\bf Case 5} Suppose $d=10$.  Then $v_0$ has at most four off-cycle neighbours and since $\Delta \le 11$, every vertex has at most four non-neighbours in the neighbourhood of any one of its neighbours.\\
{\bf Subcase 5.1} Suppose $v_0$ has four off-cycle neighbours. Then there is some $v_j$ in $S$ such that $j \ne 2$ such that $v_j$ is adjacent with at least two off-cycle neighbours. Since $G$ is locally Dirac $\{v_1,v_{t-1}\} \sim S$ and $v_1 \sim v_{t-1}$. By Lemmas \ref{nonextendable1} (1), (2) and (3), $v_{j+1} \nsim \{v_1,v_{j-1}, v_0\}$ and $v_{j+1}$ is not adjacent with the off-cycle neighbours of $v_j$. Hence $v_{j+1}$ has at least five non-neighbours in $\langle N(v_j) \rangle$, contrary to Lemma \ref{nonextendable2} (3).

\noindent{\bf Subcase 5.2} Suppose $v_0$ has three off-cycle neighbours. At least two of the vertices of $S$ are adjacent with at least two off-cycle neighbours of $v_0$ and at least one of these vertices, call it $v_j$, is adjacent with $v_1$. If $v_1 \sim v_{t-1}$, then, by Lemmas \ref{nonextendable1} (1), (2) and (3), $v_{j+1} \nsim \{v_0, v_1, v_{j-1}\}$ and $v_{j+1}$ is non-adjacent with at least two off-cycle neighbours of $v_1$. By Lemma \ref{nonextendable1} (2), $j \ne 2$. So $v_{j+1}$ has five non-neighbours in $\langle N(v_j) \rangle$. By Lemma \ref{nonextendable2} (3), this is not possible. So $v_1 \nsim v_{t-1}$. Hence $\{v_1,v_{t-1}\} \sim S$. Suppose some vertex $v_j$ of $S$ is adjacent with all three off-cycle neighbours of $v_0$. Then either $j \ne 2$ or $j \ne t-2$. We consider the case where $j \ne 2$ as the other case can be argued similarly. By Lemmas \ref{nonextendable1} (1), (2) and (3), $v_{j+1}$ has five non-adjacencies:  $v_1,v_{j-1}$ and three off-cycle neighbours of $v_j$; contrary to Lemma \ref{nonextendable2} (3). So every vertex of $S$ is adjacent with at most two off-cycle neighbours of $v_0$. So there are are exactly four vertices in $S$ that are adjacent with exactly two off-cycle neighbours of $v_0$ and the fifth vertex of $S$ is adjacent with one or two vertices of $S$. There are at least three vertices of $S$ adjacent with two off-cycle neighbours of $v_0$ and with $v_1$. At least two of these, call them $v_j$ and $v_k$, are not $v_2$. By Lemmas \ref{nonextendable1} (1), (2) and (3), $v_{j+1} \nsim \{v_1, v_{j-1}\}$ and $v_{j+1}$ is non-adjacent with the two off-cycle neighbours of $v_j$ that are also neighbours of $v_0$. So, by Lemma \ref{nonextendable2} (3), $v_{j+1} \sim v_0$ and hence $v_{j+1}$ is adjacent with the off-cycle neighbour of $v_0$ that is not a neighbour of $v_j$. Similarly  $v_{k+1}$ is adjacent with $v_0$ and the off-cycle neighbour of $v_0$ that is not adjacent with $v_k$. So $S$ has at least two vertices that are adjacent with exactly one off-cycle neighbour of $v_0$. From the case we are in this is not possible.

\noindent{\bf Subcase 5.3} Suppose $v_0$ has two off-cycle neighbours. Assume first that $v_1 \nsim v_{t-1}$. Since $v_1$ and $v_{t-1}$ each have at least four neighbours in $S$, $|N(v_1) \cap N(v_{t-1}) \cap S| \ge 4$. Also since $x$ and $y$ each have at least five neighbours in $S$, $|N(x) \cap N(y) \cap S| \ge 2$. Suppose there is a $v_{i_j} \in S$ adjacent with $x,y,v_1$ and $v_{t-1}$. If $i_j \not\in \{2,t-2\}$, then, by Lemmas \ref{nonextendable1} (1), (2) and (3), we have the following non-adjacencies in $\langle N(v_{i_j}) \rangle$: $v_{i_j+1} \nsim \{x,y,v_1,v_{i_j-1}\}$ and $v_{i_j-1} \nsim \{x,y,v_{t-1}, v_{i_j+1}\}$. So, by Lemma \ref{nonextendable2} (3), $v_0 \sim \{v_{i_j-1}, v_{i_j+1}\}$. Hence $x$ and $y$ are both adjacent with all vertices of $S' = S -\{v_{i_j-1}, v_{i_j+1}, v_{i_j}\}$. So every vertex of $S'$ is adjacent with all four of the vertices $x,y,v_1,$ and $v_{t-1}$.  Since $|S'|=3$, there is a vertex $v_{i_k} \in S'$ such that $i_k \not\in \{2, t-2\}$. As for $v_{i_j}$ we see that $v_{i_k-1}, v_{i_k+1} \in S$. WOLG $i_j < i_k$. So $v_{i_j -1},v_{i_j +1}$ and $v_{i_k+1}$ are distinct vertices of $S$ each of which is non-adjacent with both $v_1$ and $v_{t-1}$, contrary to the fact that $|N(v_1) \cap N(v_{t-1}) \cap S| \ge 4$. So $i_j$ is $2$ or $t-2$, say the former. By Lemmas \ref{nonextendable1} (1) and (3), $v_{i_j+1} \nsim \{v_{i_j-1},x,y\}$. So by considering $\langle N(v_{i_j}) \rangle$, we see that $v_{i_j+1}$ is adjacent with at least one of $v_0$ and $v_{t-1}$.  If $v_{i_j+1} \sim v_0$, then it follows, since $v_{i_j+1} \nsim v_1$, that $v_{i_j+1}$ is not a common neighbour of $v_1$ and $v_{t-1}$ and since $v_{i_j+1} \nsim \{x,y\}$, $S'$ must have three common neighbours of $x$ and $y$. So $S'$ contains two vertices that are common neighbours of $v_1,v_{t-1},x$ and $y$. At least one of these two vertices of $S'$ is not $v_{t-2}$. By the above this is not possible. Hence $v_{i_j+1} \nsim v_0$ and $v_{i_j+1} \sim v_{t-1}$. If $v_{t-2} \sim v_0$, then, by Lemma \ref{nonextendable1} (4), neither $x$ nor $y$ is adjacent with $v_{t-2}$ (since $v_{t-1} \sim \{v_{i_j}, v_{i_j+1}\}$). Observe that $v_{t-2} \nsim v_1$; otherwise, $v_0xv_{i_j}v_{t-1}v_{i_j+1} \overrightarrow{C}v_{t-2}v_1v_0$ is an extension of $C$. So there is a vertex in $S -\{v_{i_j}, v_{t-2}\}$ that is adjacent with all four of the vertices in $\{v_1,v_{t-1},x,y\}$, which by the above is not possible. So $v_{t-2} \nsim v_0$. Now  $S -\{v_{i_j}\}$ contains at least one additional common neighbour of of $x$ and $y$, call it $v_{i_k}$. By the above, $v_{i_k}$ is not adjacent with both $v_1$ and $v_{t-1}$. Suppose $v_{i_k}$ is adjacent with $v_1$ or $v_{t-1}$, say the former. Using Lemma \ref{nonextendable1}, we have $v_{i_k+1} \nsim \{x,y, v_1, v_{i_k-1}\}$. So, by Lemma \ref{nonextendable2} (3), $v_{i_k+1} \sim v_0$. This forces another vertex in $S-\{v_{i_j}, v_{i_k}, v_2, v_{t-2}\}$  adjacent with all four vertices in $\{x,y,v_1, v_{t-1}\}$, which, by the above, is not possible.

Assume next that $v_1 \sim v_{t-1}$. Let $S =\{v_{i_1}, v_{i_2}, \ldots, v_{i_6}\}$ where $i_1 < i_2 < \ldots < i_6$. By Lemma \ref{nonextendable1} (2), and (3), $i_1 \ne 2$ and $i_6 \ne t-2$ and if $x \sim v_{i_l}$, then $v_0 \nsim \{v_{i_l-1},v_{i_l+1}\}$. Since $x$ and $y$ are each adjacent with  at least four vertices of $S$, $|N(x) \cap N(y) \cap S| \ge 2$. Let $v_{i_j},v_{i_k} \in N(x) \cap N(y) \cap S$. Suppose $v_{i_j}$ or $v_{i_k}$ is adjacent with $v_1$ or $v_{t-1}$. We will assume $v_{i_j} \sim v_1$. All other cases can be argued similarly. By the above, $i_j+1 \ne t-1$ and $i_j-1 \ne 1$ and $v_{i_j+1} \nsim v_0$. Using these facts and Lemmas \ref{nonextendable1} (1), (2) and (3), we see that in $\langle N(v_{i_j}) \rangle$, $v_{i_j+1} \nsim \{x,y,v_1,v_{i_j-1}, v_0\}$. By Lemma \ref{nonextendable2} (3), this is not possible. So $\{v_1, v_{t-1}\} \nsim \{v_{i_j}, v_{i_k}\}$. So every vertex of $S -\{v_{i_j}, v_{i_k}\}$ is adjacent with both $v_1$ and $v_{t-1}$ and exactly one of $x$ or $y$. Since $v_{i_j}$ is adjacent with at least five vertices of $N(v_0)$ and since $v_{i_j} \nsim \{v_1, v_{t-1}\}$, it follows that $v_{i_j}$ has at least two neighbours in $S - \{v_{i_j}, v_{i_k}\}$. Let $v_{i_a} \in S-\{v_{i_j}, v_{i_k}\}$ be such that $v_{i_j} \sim v_{i_a}$. We may assume $v_{i_a} \sim x$. By Lemmas \ref{nonextendable1} (1), (2) and (3), we have the following non-adjacencies in $\langle N(v_{i_a}) \rangle$: $v_{i_a+1} \nsim\{x,v_0,v_1,v_{i_a-1}\}$ and $v_{i_a-1} \nsim \{x,v_0,v_{t-1}, v_{i_a+1}\}$. So by Lemma \ref{nonextendable2} (3), both $v_{i_a+1}$ and $v_{i_a -1}$ are adjacent with every other neighbour of $v_{i_a}$. Hence $v_{i_j} \sim \{v_{i_a+1}, v_{i_a-1}\}$. As before, we see that in $\langle N(v_{i_j}) \rangle$ we have the following non-adjacencies $v_{i_j+1} \nsim \{x,y, v_0, v_{i_j-1}, v_{i_a+1}\}$. Hence, by Lemma \ref{nonextendable2} (3),  $i_a+1=i_j-1$. Using Lemmas \ref{nonextendable1} (1), (2) and (3) and the above observation, we see that $v_{i_a+1} \nsim \{v_{t-1}, v_0, v_1, v_{i_a-1},x\}$, contrary to Lemma \ref{nonextendable2} (3).

\noindent{\bf Subcase 5.4}  $v_0$ has exactly one off-cycle neighbour $x$. Let $S =\{v_{i_1}, v_{i_2}, \ldots, v_{i_7}\}$ be the cycle neighbours of $v_0$ other than $v_1$ and $v_{t-1}$ where $i_1 < i_2 < \ldots < i_7$.\\
\noindent{\bf Subcase 5.4.1} $v_1 \nsim v_{t-1}$. Then $|N(v_1) \cap N(v_{t-1}) \cap S| \ge 3$ and $|N(x) \cap S| \ge 5$. So there is at least one vertex in $S$ adjacent with $x, v_1$ and $v_{t-1}$. Suppose first that there is exactly one such vertex, call it $v_{i_j}$. Then there are exactly three vertices in $S' =N(v_1) \cap N(v_{t-1}) \cap S$, and every vertex of $S-S'$ must be adjacent with $x$ and exactly one of $v_1$ and $v_{t-1}$. Suppose $v_{i_j}$ is adjacent with a vertex $S-S'$, say $v_{i_a}$. We may assume $v_{i_a}$ is adjacent with $x$ and $v_1$. The case where $v_{i_a} \sim \{x,v_{t-1}\}$ can be argued similarly. Suppose first that $i_a < i_j$. By Lemmas \ref{nonextendable1} (1), (2) and (3) , $v_{i_j+1} \nsim \{x, v_1, v_{i_j-1}\}$. Since $v_{i_j+1}$ is non-adjacent with $x$, it is not in $S-S'$ and since it is not adjacent with both $v_1$ and $v_{t-1}$ it is not in $S'$. So $v_{i_j+1} \nsim v_0$. So $v_{i_j+1}$ has four non-adjacencies in $\langle N(v_{i_j}) \rangle$. By Lemma \ref{nonextendable2} (3), it follows that $v_{i_j+1} \sim v_{i_a}$. Using similar reasoning we now see that $v_{i_a+1} \nsim \{x, v_1, v_{i_a-1}, v_0,v_{i_j+1}\}$. This contradicts Lemma \ref{nonextendable2} (3) unless $i_a=2$. Moreover, $v_{i_a+1} \sim v_{i_j}$. Again using Lemmas \ref{nonextendable2} (1) - (4), and the case we are considering, we see that $v_{i_j-1} \nsim \{x, v_1, v_{t-1}, v_{i_j+1}, v_0\}$. This contradicts Lemma \ref{nonextendable2} (3) unless $i_j=t-2$. Moreover, $v_{i_j-1} \sim \{v_{i_a}, v_{i_a+1}\}$. By assumption $v_{i_a} \nsim v_{i_j+1} (=v_{t-1})$. By Lemmas \ref{nonextendable1} (1) - (4) we also see that $v_{i_j+1} \nsim \{x, v_1, v_{i_j-1}, v_{i_a+1}\}$. This contradicts Lemma \ref{nonextendable2} (3). So $i_a > i_j$. Since $v_{i_a} \nsim v_{t-1}$, $i_a \ne t-2$. Assume first that $i_j \ne 2$. By Lemmas \ref{nonextendable1} (1), (2) and (3), $v_{i_j-1} \nsim \{x, v_{i_j+1}, v_{t-1}\}$. Since $v_{i_j-1}$ is not adjacent with both $v_1$ and $v_{t-1}$, $v_{i_j-1} \not\in S'$ and since $v_{i_j-1} \nsim x$, $v_{i_j-1} \not\in (S-S')$. Hence $v_{i_j-1} \nsim v_0$. So, by Lemma \ref{nonextendable2}(3), it follows that $v_{i_j-1} \sim v_{i_a}$. In a similar manner we see that $v_{i_a+1} \nsim \{x, v_0, v_1, v_{i_a-1}\}$ and hence $v_{i_a+1} \sim v_{i_j}$. We can now argue in a similar manner that $v_{i_j+1} \nsim \{x, v_1, v_0, v_{i_j-1}, v_{i_a+1}\}$.  This produces a contradiction to Lemma \ref{nonextendable2} (3). So $i_j=2$. By Lemmas \ref{nonextendable1} (1) - (4), $v_{i_a+1} \nsim \{x, v_1, v_0, v_{i_a-1}\}$. So $v_{i_a+1} \sim v_{i_j}$. Since $v_{i_j+1} \nsim \{x, v_1, v_0, v_{i_a+1}\}$, it follows that $v_{i_j+1} \sim v_{i_a}$. But now $v_{i_a+1}$ has five non-adjacencies in $\langle N(v_{i_a} \rangle$, namely $v_{i_a+1} \nsim \{x, v_1, v_0, v_{i_a-1}, v_{i_j+1}\}$ unless $v_{i_a-1} = v_{i_j+1}$. So $v_{i_a+1} \sim v_{i_j}$. But now $v_{i_j+1}$ has five distinct non-adjacencies in $\langle N(v_{i_j}) \rangle$, namely, $v_{i_j+1} \nsim \{x,v_1, v_0, v_{t-1}, v_{i_a+1} \}$. Hence $v_{i_j} \nsim (S-S')$.

By Lemma \ref{nonextendable2} (3), $v_{i_j} \sim S'-\{v_{i_j}\}$. Let $S'-\{v_{i_j}\} =\{v_{i_l}, v_{i_k}\}$, where $i_l < i_k$. Assume first that $i_l<i_j<i_k$. Observe, by Lemma \ref{nonextendable1} (2), that $i_l+1 \ne i_j$ and $i_j+1 \ne i_k$. Using Lemmas \ref{nonextendable1} (1), (2) and (3) and the case we are considering, we see that $v_{i_j+1} \nsim \{x, v_1, v_0,v_{i_j-1}\}$. So, by Lemma \ref{nonextendable2} (3), $v_{i_j+1} \sim \{v_{i_l}, v_{t-1}\}$. Similarly $v_{i_j-1} \sim \{v_1, v_{i_k}\}$. Observe that $v_{i_k-1} \nsim v_{t-1}$; otherwise, $v_0xv_{i_k} \overrightarrow{C}v_{i_k-1}v_{t-1} \overleftarrow{C}v_{i_k}v_{i_j-1} \overleftarrow{C}v_1v_0$ is an extension of $C$. By Lemma \ref{nonextendable2} (5), $v_{i_k-1} \nsim \{v_1, v_{i_j-1}\}$. Also since $v_{i_k} \sim v_1$, it follows from Lemma \ref{nonextendable1} (2) that $v_{i_k-1} \nsim x$. So from the case we are in $v_{i_k-1} \nsim v_0$. So, by Lemma \ref{nonextendable2} (3), $v_{i_k-1} \sim v_{i_j}$. But now $v_0xv_{i_j}v_{i_k-1} \overleftarrow{C} v_{i_j+1}v_{t-1} \overleftarrow{C} v_{i_k}v_{i_j-1} \overleftarrow{C}v_1v_0$ is an extension of $C$.

So we may assume $i_l$ and $i_k$ are either both larger or both smaller than $i_j$, say the former.  The case where both are smaller can be argued similarly. Assume first that $i_j \ne 2$. Since $v_{i_j} \sim \{x, v_0\}$ and $v_{i_l} \sim v_1$, it follows from Lemma \ref{nonextendable1} (2) that $i_l \ne i_j+1$. By Lemmas \ref{nonextendable1} (1), (2) and (3), $v_{i_j-1} \nsim \{x, v_{t-1}, v_{i_j+1}\}$. So, from the case we are in, we see that $v_0 \nsim v_{i_j-1}$. Thus, by Lemma \ref{nonextendable2} (3), $v_{i_j-1} \sim \{v_1, v_{i_l}, v_{i_k}\}$. Similarly $v_{i_j+1} \sim \{v_{t-1}, v_{i_l}, v_{i_k}\}$. By Lemma \ref{nonextendable2} (5), $v_{i_l-1} \nsim \{v_1, v_{i_j-1}\}$. Since $v_{i_l} \sim v_1$, it follows that $v_{i_l-1} \nsim x$, by Lemma \ref{nonextendable1} (2). So $v_{i_l-1} \nsim v_0$. Also $v_{i_l-1} \nsim v_{t-1}$; otherwise, $v_0xv_{i_j} \overrightarrow{C} v_{i_l-1} v_{t-1} \overleftarrow{C} v_{i_l}v_{i_j-1} \overleftarrow{C}v_1v_0$ is an extension of $C$. Hence, by Lemma \ref{nonextendable2} (3), $v_{i_l-1} \sim v_{i_j}$. But now $v_0xv_{i_j}v_{i_l-1} \overleftarrow{C}v_{i_j+1}v_{t-1} \overleftarrow{C}v_{i_l}v_{i_j-1} \overleftarrow{C}v_1v_0$ is an extension of $C$.

Hence $i_j=2$. By Lemma \ref{nonextendable2} (4), $v_{i_l-1} \nsim v_1$. Since $v_{i_l} \sim v_1$ and $v_0 \sim x$, it follows from Lemma \ref{nonextendable1} (2), that $v_{i_l-1} \nsim x$ and hence from the case we are in $v_{i_l-1} \nsim v_0$. By Lemma \ref{nonextendable2} (4), $v_{i_l-1} \nsim v_{i_l+1}\}$. Also $v_{i_l-1} \nsim v_{t-1}$; otherwise, $v_0xv_{i_j} \overrightarrow{C}v_{i_l-1}v_{t-1} \overleftarrow{C} v_{i_l}v_1v_0$ is an extension of $C$. So, by Lemma \ref{nonextendable2} (3), $v_{i_l-1} \sim v_{i_j}$. By Lemmas \ref{nonextendable1} (1) and (2) and the case we are in $v_{i_j+1} \nsim \{x, v_1, v_0\}$. Also $v_{i_j+1} \nsim v_{t-1}$; otherwise, $v_0xv_{i_j}v_{i_l-1}\overleftarrow{C}v_{i_j+1} v_{t-1} \overleftarrow{C}v_{i_l}v_1v_0$ is an extension of $C$. So $v_{i_j+1} \sim \{v_{i_l}, v_{i_l-1}\}$. Observe that $v_{i_j+1} \nsim v_{i_l+1}$; otherwise, $v_0xv_{i_j}v_{i_l-1} \overleftarrow{C} v_{i_j+1} v_{i_l+1} \overrightarrow{C} v_{t-1}v_{i_l}v_1v_0$ is an extension of $C$. Using this fact and reasoning as before, we see that $v_{i_l+1} \nsim \{v_1, v_0, v_{i_j+1}, v_{i_l-1}\}$. So if $i_j+1 \ne i_l-1$, then $v_{i_l+1} \sim v_{i_j}$. However then $v_{i_j-1} (=v_1)$ has five non-adjacencies in $\langle N(v_{i_j}) \rangle$, namely, $v_{i_j-1} \nsim \{x, v_{t-1}, v_{i_j+1}, v_{i_l-1}, v_{i_l+1}\}$. Hence $i_l-1 =i_j+1 =3$. If $i_k \ne t-2$, we can show, using the adjacencies for $v_{i_j}$ and $v_{i_k}$, that $v_{i_j-1}$ has five non-adjacencies in $\langle N(v_{i_j}) \rangle$.

So $i_k=t-2$. Since we have already shown that $v_{i_j+1}$ has four non-adjacencies in $\langle N(v_{i_j}) \rangle$, namely $v_{i_j+1} \nsim \{x, v_1, v_0, v_{t-1}\}$, we have $v_{i_j+1} \sim v_{i_k}$. By Lemma \ref{nonextendable1} (4) and \ref{nonextendable2} (4), $v_{i_k-1} \nsim \{v_1, v_{i_k+1}(=v_{t-1})\}$. From the case we are in, we see that $v_{i_k-1} \nsim v_0$. If $v_{i_k-1} \sim v_{i_j+1}$, then $v_0xv_{i_j}v_{t-1}v_{i_k}v_{i_j+1}v_{i_k-1} \overleftarrow{C}v_{i_l}v_1v_0$ is an extension of $C$. So $v_{i_k-1}$ has four non-adjacencies in $\langle N(v_{i_k}) \rangle$. Since $\Delta =11$, it follows from Lemma \ref{nonextendable2} (3), that $v_{i_k-1} \sim v_{i_j}$. But now $v_{i_j+1}$ has five non-adjacencies in $\langle N(v_{i_j}) \rangle$, namely, $v_{i_j+1} \nsim \{x, v_1, v_0, v_{t-1}, v_{i_k-1}\}$.

So we conclude that $|N(v_1) \cap N(v_{t-1}) \cap N(x) \cap S| \ge 2$.
Let $T=\{x,v_1,v_{t-1}\}$.  Assume first that each vertex of $S$ is adjacent with at least one vertex of $T$. Assume next that $\{x,v_0\} \sim \{v_2, v_{t-2}\}$. Assume also that $S-\{v_2,v_{t-2}\}$ contains a vertex $v_{i_j}$ such that $v_{i_j} \sim T$. By Lemmas \ref{nonextendable1} (1), (2) and (4), $v_{i_j-1} \nsim \{x,v_1,v_{t-1}, v_{i_j+1}\}$. Since $v_{i_j -1} \nsim \{x,v_1,v_{t-1}\}$, it follows from the case we are in that $v_{i_j-1} \not\in S$; so $v_{i_j-1} \nsim v_0$, contrary to Lemma \ref{nonextendable2} (3).

So every vertex of $S-\{v_2, v_{t-2}\}$ is adjacent with at most two vertices of $T$. By the case we are in, it thus follows that $\{v_2, v_{t-2}\} \sim T$. Moreover, there is at most one vertex of $S-\{v_2, v_{t-2}\}$ that is adjacent with exactly one vertex of  $T$.  There exist vertices $v_{i_q}, v_{i_r},v_{i_s} \in S-\{v_2, v_{t-2}\}$ such that $v_{i_q} \sim \{x,v_1\}$, $v_{i_r} \sim \{x, v_{t-1}\}$ and $v_{i_s} \sim \{v_1, v_{t-1}\}$. Let $v_{i_a}$ and $v_{i_b}$ be the vertices of $S-\{v_2,v_{t-2}, v_{i_q}, v_{i_r}, v_{i_s}\}$. At least one of these two vertices is adjacent with exactly two vertices of $T$, say $v_{i_a}$ is such a vertex. We show next that $v_{i_q} \nsim \{v_{i_r}, v_{i_s}\}$.

Assume first that $v_{i_q} \sim v_{i_r}$. We consider the case where $i_q < i_r$. The case where $i_q > i_r$ can be argued similarly. By Lemmas \ref{nonextendable1} (1) and (4), $\{v_{i_q-1}, v_{i_r+1}\} \nsim \{x,v_1,v_{t-1}\}$. From the case we are in, it follows that $v_0 \nsim \{v_{i_q-1}, v_{i_r+1}\}$.  By Lemma \ref{nonextendable1} (3), $v_{i_q-1} \nsim v_{i_q+1}$ and $v_{i_r+1} \nsim v_{i_r-1}$. So $v_{i_q-1}$ ($v_{i_r+1}$), has four non-adjacencies in $\langle N(v_{i_q}) \rangle$, ( $\langle N(v_{i_r}) \rangle$, respectively). So, by Lemma \ref{nonextendable2} (3), $v_{i_q-1} \sim v_{i_r}$ and $v_{i_r+1} \sim v_{i_q}$. By another application of Lemma \ref{nonextendable2} (3), it follows that $v_{i_r+1} \sim v_{i_q-1}$. Now we see that $C$ has an extension, namely, $v_0xv_2 \overrightarrow{C} v_{i_q-1}v_{i_r+1} \overrightarrow{C} v_{t-1}v_{i_r} \overleftarrow{C} v_{i_q}v_1v_0$, a contradiction. So $v_{i_q} \nsim v_{i_r}$.

Suppose $v_{i_q} \sim v_{i_s}$. We assume $i_q < i_s$. The case where $i_q > i_s$ can be argued similarly. By Lemma \ref{nonextendable1} (2), $i_s > i_q+1$. By Lemmas \ref{nonextendable1} (1), (2) and (4), $v_{i_q-1} \nsim \{x, v_1, v_{t-1}\}$. So, by the case we are in, $v_{i_q-1} \nsim v_0$. By Lemma \ref{nonextendable1} (3), $v_{i_q-1} \nsim v_{i_q+1}$.  So $v_{i_q-1} \nsim \{x,v_0,v_1,v_{i_q+1}\}$ and hence by Lemma \ref{nonextendable2} (3), $v_{i_q-1}$ is adjacent with every other neighbour of $v_{i_q}$. So $v_{i_q-1} \sim v_{i_s}$. We now consider non-adjacencies of $v_{i_s+1}$ in $\langle N(v_{i_s}) \rangle$. By Lemma \ref{nonextendable1} (4), $v_{i_q+1} \nsim \{v_1,v_{t-1}\}$. Since $v_{i_s} \sim v_{t-1}$, it follows from Lemma \ref{nonextendable1} (2), that $v_{i_s+1} \nsim x$. Thus, from the case we are in, $v_{i_s+1} \nsim v_0$. By Lemma \ref{nonextendable2} (4), $v_{i_s+1} \nsim v_{i_s-1}$. Hence $v_{i_s+1}$ has four non-adjacencies in $\langle N(v_{i_s}) \rangle$. By Lemma \ref{nonextendable2} (3) and since $\Delta \le 11$, $v_{i_s+1} \sim v_{i_q-1}$. Hence $v_0xv_2 \overrightarrow{C}v_{i_q-1}v_{i_s+1} \overrightarrow{C}v_{t-1}v_{i_s} \overleftarrow{C}v_{i_q}v_1v_0$ is an extension of $C$ which is not possible. Hence $v_{i_q} \nsim v_{i_s}$.

We now show that $v_{i_q} \nsim v_{i_a}$. If $v_{i_a}$ is adjacent with $\{x,v_{t-1}\}$ or $\{v_1,v_{t-1}\}$, this follows from the above. Suppose $v_{i_a} \sim \{x,v_1\}$. WOLG may assume $i_q < i_a$. We can argue as in the previous case that $v_{i_q-1} \nsim \{x,v_0,v_1,v_{i_q+1}\}$. So, by Lemma \ref{nonextendable2} (3), $v_{i_q-1} \sim v_{i_a}$. Similarly $v_{i_a+1} \nsim \{x,v_0,v_1,v_{i_a-1}\}$ and so $v_{i_a+1} \sim \{v_{i_q}, v_{i_q-1}\}$. Observe that by Lemma \ref{nonextendable1} (1), $i_a \ne i_q+1$. We can argue as for $v_{i_q-1}$, that $v_{i_q+1} \nsim \{x,v_0,v_1,v_{i_q-1}\}$. So by Lemma \ref{nonextendable2} (3), $v_{i_q+1} \sim v_{i_a+1}$. This contradicts Lemma \ref{nonextendable1} (2). So $v_{i_q} \nsim v_{i_a}$. Therefore $v_{i_q} \nsim \{v_{i_a}, v_{i_r}, v_{i_s}, v_{t-1}\}$.  By Lemma \ref{nonextendable2} (3), $v_{i_q} \sim \{v_2, v_{i_b}, v_{t-2}\}$. As before we see that $v_{i_q-1} \nsim\{x, v_0, v_1, v_{i_q+1}\}$. So, by Lemma \ref{nonextendable2} (3), $v_{i_q-1} \sim \{v_2,v_{t-2}\}$.  By Lemmas \ref{nonextendable1}  (1) and (2), $v_{t-3} \nsim \{x,v_1,v_{t-1}, v_{i_q-1}\}$. So by Lemma \ref{nonextendable2} (3), $v_{t-3} \sim v_0$ which is not possible by the case we are considering.

So $v_2$ and $v_{t-2}$ are not both adjacent with $x$ and $v_0$. Suppose now that exactly one of $v_2$ and $v_{t-2}$, say $v_2$, is adjacent with both $x$ and $v_0$. Then there is a vertex $v_{i_j} \in S-\{v_2, v_{t-2}\}$ such that $v_{i_j} \sim T$. By Lemmas \ref{nonextendable1}  (1), (2), (3) and (4), $v_{i_j-1} \nsim \{x, v_1, v_{i_j+1}, v_{t-1}\}$.  So by Lemma \ref{nonextendable2} (3), $v_{i_j-1} \sim v_0$. This is not possible since in this case we are assuming that every vertex of $S$ is adjacent with at least one vertex of $T$.

So neither $v_2$ nor $v_{t-2}$ is adjacent with both $v_0$ and $x$.  Let $v_{i_j}, v_{i_k} \in S$ be such that $\{v_{i_j}, v_{i_k}\} \sim T$ where $2 < i_j < i_k < t-2$. By Lemmas \ref{nonextendable1} (1), (2) and (3), $v_{i_j-1} \nsim \{x, v_{i_j+1}, v_{t-1}\}$. From the case we are considering, $v_{i_j-1}$ is either adjacent with both $v_0$ and $v_1$ or is non-adjacent with both $v_0$ and $v_1$. By Lemma \ref{nonextendable2} (3), $v_{i_j-1} \sim \{v_0,v_1\}$. Similarly we can argue that $v_{i_j+1} \sim \{v_0, v_{t-1}\}$. So $S$ contains at least two vertices that are adjacent with exactly one vertex of $T$, contrary to the assumptions of the case we are in.


So there is at least one vertex of $S$ that is not adjacent with any vertex of $T$. Observe also, since each vertex of $T$ is adjacent with at least five vertices of $S$, that there are at most two vertices of $S$ that are not adjacent with any vertex of $T$.

Suppose first that there is exactly one vertex of $S$, call it $v_{i_a}$ that is not adjacent with any vertex of $T$. Let $S'=S \cap N(v_1) \cap N(v_{t-1}) \cap N(x)$. Then $|S'|$ equals $3$ or $4$.  Suppose first that $|S'| =3$. Then the vertices of $S-(S' \cup \{v_{i_a}\})$ are each adjacent with exactly two vertices of $T$. Suppose $S'=\{v_{i_j}, v_{i_k}, v_{i_l}\}$, where $i_j < i_k < i_l$, and let $S-(S' \cup \{v_{i_a}\})=\{v_{i_r},v_{i_s},v_{i_t}\}$, where $v_{i_r} \sim \{x,v_1\}$, $v_{i_s} \sim \{x, v_{t-1}\}$ and $v_{i_t} \sim \{v_1, v_{t-1}\}$.

Suppose first that $i_j=2$. By Lemma \ref{nonextendable1} (1), $i_j+1 \ne i_k$. Suppose that $j_j+1 =i_k-1$. Then, by Lemmas \ref{nonextendable1} (1), (2) and (3),  $v_{i_k-1}(=v_{i_j+1}) \nsim \{x, v_1, v_{t-1}, v_{i_k+1}\}$. So, by Lemma \ref{nonextendable2} (3), $v_{i_k-1} \sim v_0$. Hence $v_{i_k-1} = v_{i_a}$.  Again, by Lemmas \ref{nonextendable1} (1), (3) and (4), $v_{i_k+1} \nsim \{x, v_1, v_{i_k-1}\}$. So, by Lemma \ref{nonextendable2} (3), $v_{i_k+1}$ must be adjacent with at least one of $v_0$ and $v_{t-1}$. Since $v_{i_k+1}$ is not $v_{i_a}$ and from the case we are in, $v_{i_k+1} \nsim v_0$. Hence $v_{i_k+1} \sim v_{t-1}$. Thus, by Lemma \ref{nonextendable1} (4), $v_{t-2}$ is not adjacent with both $x$ and $v_0$. Hence $v_{i_l} \ne t-2$. So, by Lemmas \ref{nonextendable1}  (1) - (4) and from the case we are in, $v_{i_l-1}$ has five non-adjacencies in $\langle N(v_{i_l}) \rangle$, namely,  $v_{i_l-1} \nsim \{x, v_1, v_{t-1}, v_0, v_{i_l+1}\}$ contrary to Lemma \ref{nonextendable2} (3). Hence $i_k > i_j+2$.  By Lemmas \ref{nonextendable1}  (1) - (4), $v_{i_k-1} \nsim \{x,v_1,v_{t-1}, v_{i_k+1}\}$. So by Lemma \ref{nonextendable2} (3), $v_{i_k-1} \sim v_0$. Hence $v_{i_k-1} = v_{i_a}$. Observe that $i_l \ne t-2$; otherwise, we can argue using Lemmas \ref{nonextendable1}  (1) - (4) and the fact that $v_{i_k+1} \ne v_{i_a}$, that $v_{i_k+1}$ has five non-adjacencies in $\langle N(v_{i_k} \rangle$, namely, $v_{i_k+1} \nsim \{x, v_1, v_{t-1}, v_{i_k-1}\}$. Since $i_l \ne t-2$, we can argue using Lemmas \ref{nonextendable1}  (1) - (4), the case we are in and the fact that $v_{i_l-1} \ne v_{i_a}$, that $v_{i_l-1}$ has five non-adjacencies in $\langle N(v_{i_l}) \rangle$, namely $v_{i_l-1} \nsim \{x, v_1, v_{t-1},v_{i_l+1}, v_0\}$, contrary to Lemma \ref{nonextendable2} (3).

 Hence we may assume that $i_j \ne 2$ and similarly $i_l \ne t-2$. By Lemmas \ref{nonextendable1} (1), (2) and (3), $v_{i_j-1} \nsim \{x, v_{t-1}, v_{i_j+1}\}$. From the case we are in, $v_{i_j-1}$ is not adjacent with both $v_0$ and $v_1$, since every vertex of $S-S'$ is either adjacent with no vertex of $T$ or exactly two vertices of $T$. Using Lemma \ref{nonextendable2} (3), we conclude that $v_{i_j-1}$ is adjacent with exactly one of $v_0$ and $v_1$. Similarly $v_{i_j+1}$ is adjacent with exactly one of $v_0$ and $v_{t-1}$. The same observation can be made for the two neighbours of $v_{i_k}$ on $C$ and the two neighbours of $v_{i_l}$ on $C$. Since $S$ contains exactly one vertex that is not adjacent with any vertices of $T$, it follows that either for at least two vertices of $S'$, say $v_{i_j}$ and $v_{i_k}$ (the other cases can be dealt with in a similar manner) we have $v_0 \nsim \{v_{i_j-1}, v_{i_j+1}, v_{i_k-1}, v_{i_k+1}\}$ or $v_{i_j+1} = v_{i_k-1}$ and $v_{i_j+1} \sim v_0$ or $v_{i_k+1}=v_{i_l-1}$ and $v_{i_k+1} \sim v_0$. In the first case, $v_1 \sim \{v_{i_j-1},v_{i_k-1}\}$ and $v_{t-1} \sim \{v_{i_j+1}, v_{i_k+1}\}$. Since $v_{i_j+1} \sim v_{t-1}$ and $v_{i_k-1} \nsim v_{t-1}$, it follows that $i_j+1 \ne i_k-1$. This contradicts Lemma \ref{nonextendable2} (6) (i) (where ($i=0$, $j=i_j$ and $k=i_k$) this is not possible. In the second case, we assume first that $v_{i_k-1}=v_{i_j+1}$ and $v_{i_j+1} \sim v_0$. In this case $v_{i_k+1} \sim v_{t-1}$ and $v_{i_l-1} \sim v_1$. This again contradicts Lemma \ref{nonextendable2} (6) (i) (with $i=0$, $j=i_j$ and $k=i_k$). (The case where $v_{i_k+1} =v_{i_l-1}$ and $v_{i_k+1}  \sim v_0$ can be proven similarly.)

Suppose $|S'|=4$. In this case $S-(S' \cup\{v_{i_a}\})$ contains two vertices, one of these being adjacent with two vertices of $T$ and the other being adjacent with the third vertex of $T$.  Then there exist two vertices $v_{i_j}, v_{i_k} \in S'-\{v_2,v_{t-2}\}$, where $i_j < i_k$. By Lemma \ref{nonextendable1} (1), $i_j+1 \ne i_k$. Suppose now that $i_j+2=i_k$.  By Lemmas \ref{nonextendable1}  (1), (2) and (3), $v_{i_j+1} \nsim \{x, v_1, v_{t-1}, v_{i_j-1}\}$. So, by Lemma \ref{nonextendable2} (3), $v_{i_j+1} \sim v_0$. Hence $v_{i_j+1}$ is the vertex $v_{i_a}$ of $S$ that is not adjacent with any vertex of $T$. By Lemmas \ref{nonextendable1} (1), (2) and (3), $v_{i_j-1} \nsim \{x, v_{t-1}, v_{i_j+1}\}$. If $v_{i_j-1} \sim v_0$, then from the case we are in and the above observation,  $v_{i_j-1} \sim v_1$. If $v_{i_j-1} \nsim v_0$, then $v_{i_j-1} \sim v_1$, by Lemma \ref{nonextendable2} (3). So in either case we see that $\{v_{i_j-1}, v_{i_j}\} \sim v_1$. So by Lemma \ref{nonextendable1} (4), $v_2$ is not adjacent with $x$. Similarly we can show that $\{v_{i_k},v_{i_k+1} \} \sim v_{t-1}$ and hence that $v_{t-2}$ is not adjacent with  $x$. So there is an $v_{i_l} \in S'-\{v_{i_j}, v_{i_k}\}$ such that $i_l \not\in \{2, t-2\}$. So either $2 < i_l < i_j$ or $i_k < i_l <t-2$. We may assume $2<i_l < i_j$. The case where $t-2> i_l > i_k$ can be argued similarly. From the case we are considering and by the above observation, $i_l+1 \ne i_j-1$.  By Lemmas \ref{nonextendable1}  (1), (2) and (3), $v_{i_l+1} \nsim \{x, v_1, v_{i_l-1}\}$. From the case we are considering and by Lemma \ref{nonextendable2} (3), we see that $v_{i_l+1} \sim v_{t-1}$, regardless whether $v_{i_l+1}$ is adjacent with $v_0$ or not. As before, this contradicts Lemma \ref{nonextendable2} (6). Hence $i_k > i_j+2$. So $v_{i_j}$ or $v_{i_k}$ is not adjacent with $v_{i_a}$ on $C$. We may assume $v_{i_a} \not\in \{v_{i_j-1}, v_{i_j+1}\}$. (The case where $v_{i_a} \not\in \{v_{i_k-1}, v_{i_k+1}\}$ can be argued similarly.) By Lemmas \ref{nonextendable1} (1), (2) and (3) $v_{i_j-1} \nsim \{x, v_{t-1}, v_{i_j+1}\}$. So by Lemma \ref{nonextendable2} (3), $v_{i_j-1}$ is adjacent with at least one of $v_0$ and $v_1$. Since $v_{i_a} \ne v_{i_j-1}$, $v_{i_j-1}$ must be adjacent with $v_1$ regardless of whether it is adjacent with $v_0$ or not. Similarly we can argue that $v_{i_j+1}$ is adjacent with $v_{t-1}$. By Lemma \ref{nonextendable1} (4), we now see that neither $v_2$ nor $v_{t-2}$ is adjacent with both $v_0$ and $x$. So $\{v_2, v_{t-2}\} \cap S' = \emptyset$. Hence there is a vertex $v_{i_l} \in S'-\{v_{i_j}\}$ such that $v_{i_a} \not\in \{v_{i_l-1}, v_{i_l+1} \}$. We can argue as for $v_{i_j}$ that $v_{i_l-1} \sim v_1$ and $v_{i_j+1} \sim v_{t-1}$. We may assume $i_j < i_l$. By Lemma \ref{nonextendable2} (6) (i) (with $i=0$, $j=i_j$ and $k=i_l$), it now follows that $C$ is extendable, a contradiction.

Suppose now that there are exactly two vertices of $S$, say $v_{i_a}$ and $v_{i_b}$, that are not adjacent with any vertex of $T$. Then every vertex of $S -\{v_{i_a}, v_{i_b}\}$ is adjacent with every vertex of $T$. Let $S' =S-\{v_{i_a}, v_{i_b}\} = \{v_{i_1}, v_{i_2}, \ldots, v_{i_5}\}$ where $i_1 < i_2 < \ldots < i_5$. Then there is a $v_{i_j-1} \in \{v_{i_2-1}, v_{i_3-1}, v_{i_4-1}\} - \{v_{i_a},v_{i_b}\}$. By Lemmas \ref{nonextendable1}  (1), (2) and (3), $v_{i_j-1} \nsim \{x, v_{t-1}, v_{i_j+1}\}$. By our choice of $v_{i_j-1}$ and the case we are in, $v_{i_j-1} \nsim v_0$. So by Lemma \ref{nonextendable2} (3), $v_{i_j-1} \sim v_1$. By Lemma \ref{nonextendable1} (4), it follows that $v_2$ is not adjacent with $x$. So $i_1 >2$. Similarly there is a vertex $v_{i_k+1} \in  \{v_{i_2+1}, v_{i_3+1}, v_{i_4+1}\} - \{v_{i_a},v_{i_b}\}$ such that $v_{i_k+1} \sim v_{t-1}$. So again by Lemma \ref{nonextendable1} (4), $v_{t-2}$ is not adjacent with $x$. Hence $i_5 < t-2$. So there is an $v_{i_j} \in S'$ such that $\{v_{i_j-1}, v_{i_j+1}\} \cap \{v_{i_a}, v_{i_b}\} = \emptyset$. We can argue as before that $v_{i_j-1} \sim v_1$ and $v_{i_j+1} \sim v_{t-1}$. Moreover, there is either an $i_k < i_j$ such that $v_{i_k+1} \not\in \{v_{i_a}, v_{i_b}\}$ or an $i_k > i_j$ such that $v_{i_k-1} \not\in \{v_{i_a},v_{i_b}\}$. In the first case we can show as before that $v_{i_k+1} \sim v_{t-1}$ and in the second case $v_{i_k-1} \sim v_1$. In either case we obtain, as before, a contradiction to Lemma \ref{nonextendable2} (6). So $v_1 \sim v_{t-1}$.

\noindent{\bf Subcase 5.4.2} $v_1 \sim v_{t-1}$. Let $T=\{x, v_1, v_{t-1}\}$ and $S = N(v_0)-T$. Since $G$ is locally Dirac, there are at least 13 edges joining vertices of $T$ with vertices of $S$. Moreover at least five of these edges are incident with $x$ and at least four edges are incident with each of $v_1$ and $v_{t-1}$. So $S$ contains a common neighbour $v_j$ of $x, v_0$ and $v_1$. By Lemma \ref{nonextendable1} (2), $j \not\in \{2, t-2\}$ and by Lemma \ref{nonextendable1} (3),  $v_{j-1} \nsim v_{j+1}$. Since $v_1 \sim v_{t-1}$ and $x \sim \{v_0, v_j\}$, it follows from Lemma \ref{nonextendable1} (3), that $v_0 \nsim v_{j+1}$. By Lemmas \ref{nonextendable1} (1), (2) and (3), we also see that $v_{j+1} \nsim \{x, v_1, v_{j-1}\}$. Hence $v_{j+1}$ has four non-adjacencies in $\langle N(v_j) \rangle$. So by Lemma \ref{nonextendable2} (3) $deg(v_j) \ge 10$. Hence $v_j$ is another cycle vertex adjacent with an off-cycle neighbour and having maximum degree. Since $v_{j-1} \nsim v_{j+1}$ we can argue as we did for $v_0$ that this is not possible.
Hence $d > 10$.

\noindent{\bf Case 6} $d=11$.  This case can be argued in a similar manner to Case 5 and is included in the Appendix.
\end{proof}

\section{Concluding Remarks}
In this paper we studied the structure, connectivity and edge-connectivity as well as the cycle structure of locally Dirac and Ore graphs. It follows from the work done in \cite{HK} that locally Dirac graphs are hamiltonian as well as $\{1,2\}$-extendable. The results from Section 3 suggest that these graphs have an even richer cycle structure. Indeed these results lend supporting evidence to Ryj\'{a}\v{c}ek's conjecture. However, it remains on open problem to determine whether Ryj\'{a}\v{c}ek's conjecture holds for all locally Dirac graphs.

\section{Appendix: Proof of Case 6 of Theorem \ref{cycle_extendability_in_loc_Dirac}}

\noindent{\bf Case 6} $d=11$.
Let $x$ be an off-cycle neighbour of $v_0$. Let $T = \{x,v_1,v_{t-1}\}$ and $S=N(v_0) - T$.  \\
{\bf Subcase 6.1} Assume first that $v_1 \nsim v_{t-1}$. Since $G$ is locally Dirac, there exist at least 18 edges joining vertices of $T$ with vertices of $S$. So there exists at least two vertices of $S$ that are adjacent with every vertex of $T$.

Assume first that $\{x, v_0\} \sim \{v_2, v_{t-2}\}$.  Suppose there exists a $v_{j} \in S-\{v_2,v_{t-2}\}$ such that $v_{j} \sim T$. Then, by Lemmas \ref{nonextendable1} (1) - (4), $v_{j-1} \nsim \{x, v_1, v_{j+1}, v_{t-1}\}$ and $v_{j+1} \nsim \{x,v_1,v_{j-1}, v_{t-1}\}$. So by Lemma \ref{nonextendable2} (3), $\{v_{j-1}, v_{j+1}\} \sim v_0$. So there exist at least two neighbours of $v_0$ that are not adjacent with any vertex of $T$. But then all vertices of $S'=S-\{v_{j-1}, v_{j+1}\}$ are adjacent with all three vertices of $T$.  Let $v_{k} \in S'-\{v_{j}\}$. We may assume $k > j$. Then as for $v_{j+1}$ we can show that $v_{k+1} \nsim \{x,v_1,v_{t-1},v_{k-1}\}$. Hence $v_{k+1} \sim v_0$. So $S$ has three vertices none of which are adjacent with any vertex of $T$. This is not possible. So we may assume that $v_2$ and $v_{t-2}$ are the only vertices of $S$ adjacent with all three vertices of $T$. Hence all vertices of $S-\{v_2, v_{t-2}\}$ must be adjacent with exactly two vertices of $T$ and hence lie on $C$. So there are four vertices of $S-\{v_2,v_{t-2}\}$ adjacent with $x$ and exactly one of $v_1$ and $v_{t-1}$ and there exist two vertices in $S-\{v_2, v_{t-2}\}$ adjacent with $v_1$ and $v_{t-1}$ but not with $x$. Since $G$ is locally Dirac, $v_2$ is adjacent with at least two vertices of $S -\{v_{t-2}\}$. Let $v_{j}$ be a neighbour of $v_2$ in $S-\{v_2, v_{t-2}\}$. We consider three cases. Suppose first that  $v_{j} \sim (T-\{v_{t-1}\})$. By Lemmas \ref{nonextendable1} (1), (2) and (3) and the above observation, $v_{j+1} \nsim \{x, v_1, v_0,  v_{j-1}\}$. So by Lemma \ref{nonextendable2} (3), $v_{j+1} \sim v_2$. We now see that $v_3$ has five non-adjacencies  in $\langle N(v_2) \rangle$, namely, $v_3 \nsim \{x, v_1,v_{t-1}, v_0, v_{j+1}\}$ which is not possible. So this case cannot occur. Suppose next that $v_{j}  \sim (T-\{v_1\})$. This time we can show that $v_{j+1} \nsim \{x, v_{t-1}, v_0, v_{j-1}\}$. So $v_{j+1} \sim v_2$. Since $v_3 \nsim \{x,v_1, v_{t-1}, v_0\}$, it follows from Lemma \ref{nonextendable2} (3) that $v_3 \sim v_{j+1}$ which contradicts Lemma \ref{nonextendable1} (2). Lastly assume $v_{j} \sim T-\{x\}$. Then $j >3$. From the cases we have considered and since $G$ is locally Dirac we see that $v_2 \sim v_{t-2}$ and $v_{t-2} \sim v_{j}$. By Lemma \ref{nonextendable1} (4), $v_{j+1} \nsim \{v_1,v_{t-1}\}$ and thus by the above observation, $v_{j+1} \nsim v_0$. By Lemma \ref{nonextendable2} (4), $v_{j+1} \nsim v_{j-1}$.  So, by Lemma \ref{nonextendable2} (3), $v_{j+1} \sim \{v_2, v_{t-2}\}$.  Similarly $v_{j-1} \sim \{v_2, v_{t-2}\}$. As before we can argue that $v_3 \nsim \{x, v_1, v_{t-1}, v_0\}$ and hence $v_{3} \sim \{v_{j-1}, v_{j+1}\}$. But now $v_0xv_2v_{j-1} \overleftarrow{C}v_3v_{j+1} \overrightarrow{C}v_{t-1}v_{j}v_1v_0$ is an extension of $C$ which is not possible.

So either $v_2$ or $v_{t-2}$, say $v_{t-2}$,  is not adjacent with both $v_0$ and $x$.  Assume first that $v_2 \sim \{x, v_0\}$. Then there is a $v_{j} \in S-\{v_2\}$ such that $v_{j} \sim T$ and $j \ne t-2$. By Lemma \ref{nonextendable1} (2), $j >3$. By Lemmas \ref{nonextendable1} (1) - (4), $v_{j-1} \nsim \{x,v_1,v_{j+1}, v_{t-1}\}$. So by Lemma \ref{nonextendable2} (3), $v_{j-1} \sim v_0$. But then there exist at least four vertices in $S$ adjacent with every vertex of $T$ and hence at least three vertices in $S-\{v_2\}$ adjacent with all vertices of $T$. However, then there exist at least three vertices of $S$ not adjacent with any vertex of $T$ which is not possible. So neither $v_2$ nor $v_{t-2}$ is adjacent with both $x$ and $v_0$.

Let $v_{j}, v_{k} \in S$ be vertices adjacent with all vertices of $T$ where $j < k$. By the above,  $2 < j< k < t-2$. Suppose that these are the only vertices of $S$ that are adjacent with every vertex of $T$. By an earlier observation, the remaining vertices of $S$ are necessarily adjacent with exactly two vertices of $T$. By Lemmas \ref{nonextendable1} (1), (2) and (3), $v_{j+1} \nsim \{x, v_1, v_{j-1}\}$. By our observation, $v_{j+1} \nsim v_0$. Hence by Lemma \ref{nonextendable2} (3), $v_{j+1} \sim v_{t-1}$. Similarly $v_{k-1} \sim v_1$. By Lemma \ref{nonextendable2} (6) (i) (with $i=0$), $C$ is extendable which is not possible. So there exists at least three vertices of $S$ that are adjacent with all three vertices of $T$.  If there exists exactly three vertices of $S$ that are adjacent with all three vertices of $T$, then there is exactly one vertex in $S$ that is adjacent with exactly one vertex of $T$. So there exist two vertices $v_{j}, v_{k} \in S$ (where $j<k$) that are adjacent with every vertex of $T$ and such that $v_{j+1} \nsim v_0$ and $v_{k-1} \nsim v_0$. Since by Lemmas \ref{nonextendable1} (1), (2) and (3), we also know that $v_{j+1} \nsim \{x, v_1, v_{j-1}\}$ and $v_{k-1} \nsim \{x, v_{t-1}, v_{k+1}\}$, it follows that $v_{j+1} \sim v_{t-1}$ and $v_{k-1} \sim v_1$. So by Lemma \ref{nonextendable2} (6) (i) (with $i=0$) $C$ is extendable. So we may assume that $S$ contains at least four vertices that are adjacent with every vertex of $T$. Since $S$ has at most two vertices that are adjacent with at most one vertex of $T$, there exist two vertices $v_{j}, v_{k} \in S$ (where $j<k$) that are adjacent with every vertex of $T$ and such that $v_{j+1} \nsim v_0$ and $v_{k-1} \nsim v_0$. As in the previous case, $v_{j+1} \sim v_{t-1}$ and $v_{k-1} \sim v_1$. So, by Lemma \ref{nonextendable2} (6), $C$ is extendable which is not possible.

\noindent{\bf Subcase 6.2}  $v_1 \sim v_{t-1}$. Suppose there exist $v_{j}, v_{k} \in S$ such that $\{v_{j}, v_{k}\} \sim T$ where $j < k$. Since $v_1 \sim v_{t-1}$, it follows from Lemma \ref{nonextendable1} (2) that $2 < j < k < t-2$. By Lemmas \ref{nonextendable1} (1), (2), (3) and (4), $v_{k-1} \nsim \{x, v_0, v_{k+1}, v_{t-1}\}$ and $v_{j+1} \nsim \{x, v_0, v_1, v_{j-1}\}$. Hence by Lemma \ref{nonextendable2} (3), $v_{k-1} \sim v_1$ and $v_{j+1} \sim v_{t-1}$. So, by Lemma \ref{nonextendable2} (6), $C$ is extendable, a contradiction. Suppose next that there exists exactly one vertex $v_{j} \in S$ such that $v_{j} \sim T$.  Suppose $v_{j} \sim v_{k}$ where $v_{k} \sim (T-\{v_1\})$ or $v_{k} \sim (T-\{v_{t-1}\})$. We may assume $j < k$; the case where $j > k$ can be argued similarly. As before, we see that $2 <k< j < t-2$. Assume first that $v_{k} \sim (T-\{v_1\})$.   By Lemmas \ref{nonextendable1} (1) - (4), $v_{k-1} \nsim \{x, v_{t-1}, v_0, v_{k+1} \}$. Hence, by Lemma \ref{nonextendable2}(3), $v_{k-1} \sim v_{j}$. By Lemmas \ref{nonextendable1} (1) - (4), $v_{j-1} \nsim \{x, v_0, v_{t-1}, v_{j+1},v_{k-1}\}$.  By Lemma \ref{nonextendable2}(3) this is not possible unless $v_{j+1} = v_{k-1}$. However, then $v_{j+1}$ has five non-adjcencies in $\langle N(v_{j}) \rangle$, namely $v_{j+1} \nsim \{x, v_0, v_1, v_{t-1}, v_{j-1} \}$ which is not possible. Assume next that $v_{k} \sim (T-\{v_{t-1}\})$. By Lemmas \ref{nonextendable1} (1) - (3), $v_{j-1} \nsim \{x, v_0, v_{t-1}, v_{j+1}\}$. Hence, by Lemma \ref{nonextendable2} (3), $v_{j-1} \sim v_{k}$. Similarly $v_{k+1} \nsim \{x, v_0, v_1, v_{k-1}\}$ and so $v_{k+1} \sim v_{j}$. But now $v_{j+1}$ has five non-adjacencies in $\langle N(v_{j}) \rangle$, namely, $v_{j+1} \nsim \{x,v_0,v_1,v_{j-1}, v_{k+1}\}$, contrary to Lemma \ref{nonextendable2} (3). So $v_{j}$ is not adjacent with a vertex of $S$ that is adjacent with both $x$ and at least one of $v_1$ and $v_{t-1}$.

Since $v_{j}$ is the only vertex of $S$ adjacent with every vertex of $T$, there are six vertices of $S - \{v_{j}\}$ adjacent with exactly two vertices of $T$ and one vertex adjacent with exactly one vertex of $T$. Since $S- \{v_{j}\}$ has at least five vertices adjacent with $x$ and since $G$ is locally Dirac, $v_{j}$ must be adjacent with a vertex of $S$ that is a neighbour of $x$. By the above, such a vertex is not adjacent with either $v_1$ or $v_{t-1}$. So there are two vertices of $S$ adjacent with $T-\{x\}$ and $v_{j}$ is adjacent with both of these vertices.
 Let $v_{j} \sim v_{k}$ where $v_{k} \sim T-\{x\}$.  Hence $v_{j} \sim v_{k}$ where $v_{k} \sim T-\{x\}$. Assume $j < k$. The case where $j > k$ can be argued similarly. Note that $2 < j$ and that $k \ne j+1$, by Lemma \ref{nonextendable1} (2). As before we can argue that $v_{j-1}$ and $v_{j+1}$ both have four non-adjacencies in $\langle N(v_{j}) \rangle$, namely $ v_{j-1} \nsim \{x,v_0,v_{t-1}, v_{j+1}\}$ and $v_{j+1} \nsim \{x, v_0, v_1, v_{j-1}\}$. So, by Lemma \ref{nonextendable2} (3), $v_{j+1} \sim \{v_{t-1}, v_{k}\}$ and $v_{j-1} \sim \{v_1, v_{k}\}$. We consider the non-adjacencies of $v_{k-1}$ in $\langle N(v_{k}) \rangle$. By Lemma \ref{nonextendable1} (4), $v_{k-1} \nsim v_0$ since $v_1 \sim v_{t-1}$. By Lemma \ref{nonextendable2} (6) we see that $v_{k-1} \nsim v_1$. Observe next that $v_{k-1} \nsim v_{t-1}$; otherwise, $v_0xv_{j} \overrightarrow{C} v_{k-1} v_{t-1} \overleftarrow{C} v_{k} v_{j+1} \overleftarrow{C} v_1v_0$ is an extension of $C$. Next observe that $v_{k-1} \nsim v_{j-1}$; otherwise, $v_0xv_{j}v_1 \overrightarrow{C}v_{j-1}v_{k-1} \overleftarrow{C} v_{j+1}v_{k} \overleftarrow{C}v_{t-1}v_0$ is an extension of $C$. Since $j-1 \ne 1$, we have, by Lemma \ref{nonextendable2} (3) $v_{k-1} \sim \{v_{j}, v_{j+1}\}$. But now $v_0xv_{j}v_{k-1} \overleftarrow{C}v_{j+1}v_{t-1} \overleftarrow{C}v_{k}v_{j-1} \overleftarrow{C}v_1v_0$ is an extension of $C$.

So we may assume that no vertex of $S$ is adjacent with all three vertices of $T$. Then every vertex of $S$ is adjacent with exactly two vertices of $T$ and there exist exactly three vertices in $S$ adjacent with $x$ and $v_1$; exactly three adjacent with $x$ and $v_{t-1}$ and exactly two adjacent with $v_1$ and $v_{t-1}$. We say that a vertex $v_{a}$ of $S$ is of Type 1, 2 or 3, depending on whether $v_{a}$ is adjacent with all vertices of $T-\{v_{t-1}\}$, or all vertices of $T-\{v_1\}$ or all vertices of $T-\{x\}$, respectively. We establish several facts that will aid us in completing our proof.\\
\noindent{\em Fact 1:} If $v_{j}$ and $v_{k}$ are Type 1 vertices and $v_{j} \sim v_{k}$, then $k=j+2$ or $k=j-2$.\\
\noindent{\em Proof of Fact 1.}  We assume $j < k$. The other case can be proven in the same way. (Note that since $v_1 \sim v_{t-1}$, Lemma \ref{nonextendable1} (3) guarantees that $2 < k < k < t-2$. Also, by Lemma \ref{nonextendable1} (1), $k > j+1$.) By Lemmas \ref{nonextendable1} (1), (2), and (3), $v_{j+1} \nsim \{x,v_0,v_1, v_{j-1}\}$. So, by Lemma \ref{nonextendable2} (3), $v_{j+1} \sim v_{k}$. Again, using Lemmas \ref{nonextendable1} (1), (2) and (3) we see that $v_{k+1}$ has the following non-adjacencies in $\langle N(v_{k}) \rangle$, $v_{k+1} \nsim \{x,v_0,v_1,v_{j+1}, v_{k-1}\}$. By Lemma \ref{nonextendable2} (3) this is not possible unless $v_{j+1} = v_{k-1}$, i.e. if $k = j+2$. $\Box$

\noindent{\em Fact 2:} If $v_{j}$ is a Type 1 vertex, $v_{k}$ is a Type 2 vertex and $v_{j} \sim v_{k}$, then $k = j+2$.\\
{\em Proof of Fact 2.} We show first that if $k < j$, then $C$ is extendable. As before we see that $2 < k<j<t-2$ and $k+2 \le j$. By Lemmas \ref{nonextendable1} (1), (2) and (3), $v_{k-1} \nsim \{x, v_0, v_{t-1}, v_{k+1}\}$. So, by Lemma \ref{nonextendable2} (3), $v_{k-1} \sim v_{j}$. Similarly $v_{j+1} \nsim \{x, v_0,v_1,v_{j}\}$ and hence $v_{j} \sim \{v_{k}, v_{k-1}\}$. By Lemmas \ref{nonextendable1} (1), (2) and (3), $v_{k+1} \nsim \{x, v_0, v_{k-1}, v_{j+1}\}$. Hence $v_{k+1} \sim v_{t-1}$ and similarly $v_{j-1} \sim v_1$, contrary to Lemma \ref{nonextendable2} (6).

So $j < k$ and $k \ge j+2$. As before $v_{j+1} \nsim \{x,v_0, v_1, v_{j-1}\}$ and hence $v_{j+1} \sim v_{k}$. Similarly $v_{k-1} \nsim \{x,v_0,v_{t-1}, v_{k+1}\}$ and hence $v_{k-1} \sim v_{j}$. If $k \ne j+1$, $v_{j-1}$ has four non-adjacencies in $\langle N(v_{j}) \rangle$, namely, $v_{j-1} \nsim \{x, v_0, v_{j+1}, v_{k-1}\}$. So $v_{j-1} \sim v_{k}$. Now we can show similarly that $v_{k+1} \sim \{v_{j}, v_{j-1}\}$. But now $v_{j+1}$ has five non-adjacencies in $ \langle N(v_{j}) \rangle$. $\Box$

\noindent{\em Fact 3:} If $v_{j}$ is a Type 1 vertex, then $v_{j}$ is adjacent with at most one Type 2 vertex.\\
\noindent{\em Proof of Fact 3.} From Fact 2, we know that if $v_{j}$ is adjacent with a vertex of Type 2 it must be $v_{j+2}$. $\Box$

\noindent{\em Fact 4:} If $v_{j}$ is a Type 1 vertex and $v_{j}$ is not adjacent with and Type 2 vertex, then $v_{j} \sim \{ v_{j+2}, v_{j-2}\}$ and $v_{j+2}$ and $v_{j-2}$ are both Type 1 vertices.\\
{\em Proof of Fact 4.} If $v_{j}$ is not adjacent with any of the three Type 2 vertices, then these vertices and $v_{t-1}$ are the only non-neighbours of $v_{j}$ in $\langle N(v_0) \rangle$ and so $v_{j}$ is adjacent with all remaining vertices of $S$. In particular, $v_{j}$ is adjacent with the other two Type 1 vertices, which, by Fact 1, must be $v_{j+2}$ and $v_{j-2}$. $\Box$

\noindent{\em Fact 5:} If $v_{j}$ is a Type 1 vertex that is adjacent with a Type 1 vertex $v_{l}$ and a Type 2 vertex $v_{k}$, then $v_{l} =v_{j-2}$ and $v_{k} =v_{j+2}$.\\
{\em Proof of Fact 5.} By Fact 2, $v_{k} = v_{j+2}$. By Fact 1, it now necessarily follows that $v_{l} = v_{j-2}$. $\Box$

\noindent{\em Fact 6:} If $v_{j}$ is a Type 1 vertex, then $v_{j}$ is adjacent with $v_{j+2}$ and $v_{j-2}$ and either (i) both $v_{j+2}$ and $v_{j-2}$ are Type 1 vertices or (ii) $v_{j+2}$ is a Type 2 vertex and $v_{j-2}$ is a Type 1 vertex.\\
{\em Proof of Fact 6.} By Lemma \ref{nonextendable2} (3), $v_{j}$ is non-adjacent with at most three vertices of $S$ in addition to $v_{t-1}$. By Fact 3, $v_{j}$ is adjacent with at most one Type 2 vertex. Hence $v_{j}$ is necessarily adjacent with at least one Type 1 vertex. By Fact 4, if $v_{j}$ is not adjacent with a Type 2 vertex, then it must be adjacent with two Type 1 vertices. The rest of the result follows from Facts 4 and 5. $\Box$

We now complete our proof. Let $v_{j}$ be a Type 1 vertex. By Fact 6, $v_{j} \sim \{v_{j-2}, v_{j+2}\}$. Since $x \sim \{v_{j-2}, v_{j+2}\}$ and $v_1 \sim v_{t-1}$, it follows from Lemma \ref{nonextendable1} (3) that $2 < j-2$ and $j+2 < t-2$. Suppose first that $v_{j-2}$ and $v_{j+2}$ are both Type 1 vertices. Now, by Lemmas \ref{nonextendable1} (1), (2) and (3), $v_{j-1} \nsim \{x, v_0, v_1, v_{j+1}\}$. So $v_{j-1} \sim v_{j+2}$. Now again by Lemmas \ref{nonextendable1} (1), (2) and (3), $v_{j+1}$ has five non-adjacencies in $\langle N(v_{j+2}) \rangle$, namely, $v_{j+3} \nsim \{x,v_0, v_1, v_{j+1}, v_{j-1}\}$, contrary to Lemma \ref{nonextendable2} (3). So, by Fact 6, $v_{j+2}$ is of Type 2 and $v_{j-2}$ is of Type 1. Again, by Lemmas \ref{nonextendable1} (1), (2) and (3), $v_{j-1} \nsim \{x, v_0, v_1, v_{j+1}\}$. So $v_{j-1} \sim v_{j+2}$. Now $v_{j+1}$ has five non-adjacencies in $\langle N(v_{j+2}) \rangle$, namely, $v_{j+1} \nsim \{x,v_0,v_{t-1}, v_{j-1}, v_{j+3}\}$, contrary to Lemma \ref{nonextendable2} (3).

\end{document}